\documentclass[12pt,letterpaper]{amsart}
\pdfoutput=1 

\usepackage[utf8]{inputenc}
\usepackage{amsmath}
\usepackage{amsthm}
\usepackage{amssymb}
\usepackage{mdframed}
\usepackage{indentfirst}
\usepackage{url}
\usepackage{tikz-cd}
\usepackage[colorlinks=true]{hyperref}
\usepackage{colonequals}
\usepackage{graphicx}
\usepackage{rotating}
\usepackage{enumerate}
\usepackage{enumitem}
\usepackage{soul}
\usepackage{stmaryrd}
\usepackage{mathrsfs}
\usepackage{mathtools}
\usepackage{ccaption}
\usepackage{pict2e}
\usepackage[dvips,centering,includehead,width=15.1cm,height=22.7cm,
paperheight=23.8cm,paperwidth=17cm
]{geometry}

\numberwithin{equation}{section}
\changecaptionwidth
\captionwidth{5.6in}

\newtheorem{thm}[subsection]{Theorem}
\newtheorem{lem}[subsection]{Lemma}

\newtheorem{prop}[subsection]{Proposition}

\newtheorem{cor}[subsection]{Corollary}
{
\theoremstyle{definition}
\newtheorem{rem}[subsection]{Remark}

\newtheorem{definition}[subsection]{Definition}
\newtheorem{example}[subsection]{Example}
\newtheorem{defn}[subsection]{Definition}
}

\newcommand{\CC}{\mathbb C}

\newcommand{\ZZ}{\mathbb Z}
\newcommand{\RR}{\mathbb R}
\newcommand{\AAA}{\mathbf{A}}
\newcommand{\VVV}{\mathbf{V}}
\newcommand{\xS}{\mathbf{xS}}
\newcommand{\xx}{\mathbf{x}}
\newcommand{\CM}{{Cayley-Menger} }
\newcommand{\InCM}{{I_n^{\textup{CM}}}}
\newcommand{\xdiamond}{\rotatebox[origin=c]{45}{$\boxplus\!\!$
}}

\newcommand{\neline}{{\scriptstyle\diagup}}
\newcommand{\seline}{{\scriptstyle\diagdown\,}}
\newcommand{\smallneline}{{\,\scriptscriptstyle\diagup}}
\newcommand{\smallseline}{{\scriptscriptstyle\diagdown\hspace{0.3pt}}}

\newcommand{\B}{\bigg}

\newcommand{\Lnote}[1]{{\color{red}[$\star$#1$\star$]}}

\begin{document}
\title{Heronian Friezes}

\author{Sergey Fomin}
\address{Department of Mathematics, University of Michigan, Ann Arbor, MI 48109, USA}
\email{fomin@umich.edu}

\author{Linus Setiabrata}
\address{Department of Mathematics, Cornell University, Ithaca, NY 14853, USA} 
\email{ls823@cornell.edu}

\thanks
{\emph{2010 Mathematics Subject Classification}
Primary 52C25,  
Secondary  13F60, 
51K99. 
}

\thanks{Partially supported by NSF grant DMS-1664722.}

\keywords{Heron's formula, frieze, Laurent phenomenon, Cayley-Menger equation, distance geometry, rigidity, triangulation.} 

\date{September 1, 2019. Revised December 26, 2019.}

\begin{abstract}
Motivated by computational geometry of point configurations on the Euclidean plane, and by the theory of cluster algebras of type~$A$, we introduce and study \emph{Heronian friezes}, the Euclidean analogues of Coxeter's frieze patterns. 
We prove that a generic Heronian frieze possesses the glide symmetry (hence is periodic), and establish the appropriate version of the \emph{Laurent phenomenon}. 

For a closely related family of \emph{\CM friezes}, we identify an algebraic condition of coherence, which all friezes of geometric origin satisfy. This yields an unambiguous propagation rule for coherent \CM friezes, as well as the corresponding periodicity results. 
%
\end{abstract}
\ \vspace{-.1in}

\maketitle

\vspace{-.4in}

\begin{center}
\setlength{\unitlength}{0.8pt}
\begin{picture}(400,110)(0,0)
\thicklines
\multiput(0,0)(40,0){9}{\line(1,1){80}}
\multiput(0,80)(40,0){9}{\line(1,-1){80}}

\put(0,40){\line(1,1){40}}
\put(0,40){\line(1,-1){40}}
\put(400,40){\line(-1,-1){40}}
\put(400,40){\line(-1,1){40}}

\multiput(0,0)(40,0){11}{\circle*{4}}
\multiput(10,10)(20,0){20}{\circle*{4}}
\multiput(20,20)(40,0){10}{\circle*{4}}
\multiput(10,30)(20,0){20}{\circle*{4}}
\multiput(0,40)(40,0){11}{\circle*{4}}
\multiput(10,50)(20,0){20}{\circle*{4}}
\multiput(20,60)(40,0){10}{\circle*{4}}
\multiput(10,70)(20,0){20}{\circle*{4}}
\multiput(0,80)(40,0){11}{\circle*{4}}

\thinlines
\multiput(-1.5,18.5)(5,5){13}{\color{blue}\line(1,1){3}}
\multiput(-1.5,58.5)(5,5){5}{\color{blue}\line(1,1){3}}
\multiput(18.5,-1.5)(5,5){17}{\color{blue}\line(1,1){3}}
\multiput(58.5,-1.5)(5,5){17}{\color{blue}\line(1,1){3}}
\multiput(98.5,-1.5)(5,5){17}{\color{blue}\line(1,1){3}}
\multiput(138.5,-1.5)(5,5){17}{\color{blue}\line(1,1){3}}
\multiput(178.5,-1.5)(5,5){17}{\color{blue}\line(1,1){3}}
\multiput(218.5,-1.5)(5,5){17}{\color{blue}\line(1,1){3}}
\multiput(258.5,-1.5)(5,5){17}{\color{blue}\line(1,1){3}}
\multiput(298.5,-1.5)(5,5){17}{\color{blue}\line(1,1){3}}
\multiput(338.5,-1.5)(5,5){13}{\color{blue}\line(1,1){3}}
\multiput(378.5,-1.5)(5,5){5}{\color{blue}\line(1,1){3}}

\multiput(-1.5,61.5)(5,-5){13}{\color{blue}\line(1,-1){3}}
\multiput(-1.5,21.5)(5,-5){5}{\color{blue}\line(1,-1){3}}
\multiput(18.5,81.5)(5,-5){17}{\color{blue}\line(1,-1){3}}
\multiput(58.5,81.5)(5,-5){17}{\color{blue}\line(1,-1){3}}
\multiput(98.5,81.5)(5,-5){17}{\color{blue}\line(1,-1){3}}
\multiput(138.5,81.5)(5,-5){17}{\color{blue}\line(1,-1){3}}
\multiput(178.5,81.5)(5,-5){17}{\color{blue}\line(1,-1){3}}
\multiput(218.5,81.5)(5,-5){17}{\color{blue}\line(1,-1){3}}
\multiput(258.5,81.5)(5,-5){17}{\color{blue}\line(1,-1){3}}
\multiput(298.5,81.5)(5,-5){17}{\color{blue}\line(1,-1){3}}
\multiput(338.5,81.5)(5,-5){13}{\color{blue}\line(1,-1){3}}
\multiput(378.5,81.5)(5,-5){5}{\color{blue}\line(1,-1){3}}
\end{picture}
\end{center}

\section{Introduction}
\label{sec:intro}

Coxeter's \emph{frieze patterns}~\cite{coxeter-friezes} are certain multi-line arrays of numbers satisfying a simple local condition (a~determinantal recurrence).
They arise in multiple mathematical contexts including quiver representations, plane hyperbolic geometry, and most recently, cluster algebras of type~$A$;
see~\cite{morier-genoud} for an excellent survey. 
In this paper, we introduce \emph{Heronian friezes}, the analogues of Coxeter's friezes built using recurrence relations arising in the context of metric geometry of the Euclidean plane. 
A~Heronian frieze is an algebraic abstraction of the set of measurements associated with an $n$-tuple of points on the plane; these measurements include the squared distances between pairs of points as well as signed areas of oriented triangles formed by triples of points. Just like the ordinary friezes, the Heronian ones are governed by rational recurrences. The key distinction from the classical setting 
is that the quantities being updated as one moves along a Heronian frieze are not algebraically independent: they satisfy Heron's formulas. Crucially, these algebraic dependences propagate under the frieze recurrences. 

We establish the basic properties of Heronian friezes, most importantly those concerning \emph{periodicity} and \emph{Laurentness}. We also study a related notion of a \emph{Cayley-Menger frieze}, based on the eponymous equation involving the six squared distances between four coplanar points. 
To achieve unambiguous single-valued propagation in a Cayley-Menger frieze, we identify a subtle algebraic condition of \emph{coherence}, which involves squared distances between six coplanar points. 

\medskip
\pagebreak[3]

We next provide a brief overview of the paper. 
Suppose one wants to describe a configuration of $n$ points on the Euclidean plane~$\AAA$, viewed up to the action of the group $\textnormal{Aut}(\AAA)$ of orientation-preserving rigid motions. The parameters (\emph{measure\-ments}) used in such a description must be $\textnormal{Aut}(\AAA)$-invariant. The standard approach of distance geometry is to use a subset of the squared distances between the points. Since the configuration space has dimension $2n-3$, it is natural to start by measuring some appropriately selected $2n-3$ squared distances. The simplest choice is to pick a \emph{triangulation} of a convex $n$-gon  by $n-3$ of its diagonals, view it as a graph with $n$ vertices and $2n-3$ edges, and measure the distances between the pairs of points in a configuration corresponding to the sides and diagonals of the polygon. Assuming that the configuration is sufficiently generic (namely, all diagonal lengths are nonzero), this brings the dimension down to zero; in other words, the number of configurations with the given values of those $2n-3$ measurements is finite. Unfortunately this number is exponentially large: for each triangle of the triangulation, there are two possible orientations, and each of the $2^{n-2}$ choices can be realized. 

One way to resolve this ambiguity is to add additional ``bracing'' edges to the triangulation~\cite{jt2019}. A~frieze version of this approach is developed in Section~\ref{sec:cayley-menger}, reviewed later in this introduction. In Section~\ref{sec:polygons}, we propose a different approach (inspired by classical invariant theory) which appears to allow for a better control of the computational and algebraic aspects of the problem: in addition to the $2n-3$ squared distances, we measure the \emph{signed areas} of the $n-2$ triangles of the triangulation. In other words, for each of these triangles, we choose one of the two square roots in Heron's formula. It turns out that once such choices have been made, the rest of the measurements (in particular, the squared distances for all $\binom{n}{2}$ pairs of points) can be computed using \emph{rational} recurrences. 

An explicit implementation of these recurrences leads us to the notion of a Heronian frieze, introduced in Section~\ref{sec:heron}. 
We show that a sufficiently generic Heronian frieze 
is uniquely determined by a small proportion of its entries. 
We~then prove, under the same genericity assumption, that any Heronian frieze possesses the \emph{glide symmetry}, and consequently is \emph{periodic}; see Theorem~\ref{thm:frieze-to-pattern}. 
These periodicity properties parallel the analogous properties of Coxeter-Conway friezes. 

In Section~\ref{sec:laurentness}, we establish the \emph{Laurent phenomenon} for Heronian friezes: every squared distance and every signed area of a triangle in an $n$-point configuration can be expressed as a Laurent polynomial in the initial measurement data associated with an arbitrary triangulation of an $n$-gon; see Theorems~\ref{thm:laurentness-full} and~\ref{thm:laurent-thin}. Note that the \hbox{$3n-5$} initial measurements are not algebraically independent, so there is no canonical rational function that expresses an arbitrary measurement in terms of the initial~ones. 
Curiously, the only initial measurements which appear 
in the denominators of our Laurent expressions are those corresponding to the diagonals of the initial triangulation. While the absence of the squared distances corresponding to the sides of the polygon did not come as a surprise (given a similar phenomenon in cluster theory), we see no simple conceptual explanation for the absence of signed areas in the denominators. 
Another mystery is that in spite of having the same underlying combinatorics as cluster algebras of type~$A$, this construction does not appear to fit into any (generalized) cluster algebras setup known to us. 


Section~\ref{sec:cayley-menger}  is essentially self-contained. It is devoted to an alternative construction of friezes adapted to Euclidean geometry of point configurations. 
This time, we do not use signed areas at all, keeping squared distances as the only entries of a frieze. The na\"ive idea is to use a propagation rule based on the \emph{Cayley-Menger equation} satisfied by the six squared distances between pairs of vertices of a plane quadrilateral. Unfortunately this approach immediately runs into a serious complication: unlike the Ptolemy relation used in the classical theory of friezes, the Cayley-Menger equation is \emph{quadratic} in each of the six variables, so the iterative process branches into two subcases at each step of the recurrence. (A~similar situation arises in the study of the Kashaev equation~\cite{kenyon-pemantle, leaf2019}.) To resolve the accumulating ambiguities, we employ an idea inspired by~\cite{leaf2019}: we identify an additional algebraic condition that must be satisfied by a Cayley-Menger frieze coming from a point configuration. This condition, which we call \emph{coherence} by analogy with~\cite{leaf2019}, involves $13$~frieze entries associated with a $3\times 3$ grid subpattern. The key advantage of the coherence equation is that it has degree~$1$ with respect to the rightmost (or leftmost) variable, so it can be used to set up a rational recurrence. Under this recurrence, the Cayley-Menger condition propagates, and a coherent frieze is uniquely reconstructed from the initial data, subject to certain genericity conditions. 
We later use these propagation rules to establish the glide symmetry of coherent \CM friezes, see 
Theorem~\ref{thm:CM-frieze-glide-symmetry}. 

In Section~\ref{sec:CM-vs-heronian}, we discuss  the relationship between Heronian and \CM friezes. We show that, subject to the aforementioned genericity conditions, the coherent \CM friezes are precisely the restrictions of Heronian friezes. This relationship closely resembles the one between the hexahedron equation of R.~Kenyon-R.~Pemantle~\cite{kenyon-pemantle} and Kashaev's equation. In fact, both relationships can be viewed as adaptations of \cite[Section 10]{leaf2019} to their respective contexts. 

Why does an approach employing both squared distances and signed areas produce simpler recurrences than the one that only uses squared distances? One possible explanation comes from the fact that in the case of point configurations on the plane, Cayley-Menger varieties are given by equations of degree~$3$, namely the vanishing of the mixed Cayley-Menger determinants, see~\cite{borcea2002}. By contrast, the ring of $\operatorname{SO}(2)$ invariants of a collection of several vectors is generated in degree~$2$. 

The results in this paper can be extended to other flat real geometries (such as the cylinder and the torus) by passing to the universal cover. We intend to investigate the hyperbolic and/or spherical cases in subsequent work. It would also be interesting to develop the analogues of these results for higher-dimensional geometry. 






Our work was inspired by several sources: the classical Coxeter-Conway theory of frieze patterns~\cite{conway-coxeter, coxeter-friezes}, the theory of rigidity phenomena in distance geometry (especially generic global rigidity on the plane  \cite{connelly2005, ght2010, gt2014}), classical invariant theory~\cite{weyl1946} (especially invariants of $\operatorname{SO}(2,\CC)$), the theory of cluster algebras of type~$A$~\cite{ca1, ca2} (especially their hyperbolic geometry models~\cite{cats2}), and A.~Leaf's theory of coherent solutions of the Kashaev equation~\cite{leaf2019}. 

\pagebreak[3]
\section{Triangulated polygons and Heronian diamonds}
\label{sec:polygons}

Let $\mathbf{V}$ be a two-dimensional vector space over~$\mathbb{C}$,
endowed with a symmetric inner product
$(u,v)\! \mapsto\! \langle u,v\rangle$
and an associated skew-symmetric volume form 
\hbox{$(u,v)\! \mapsto\! [u,v]$}. 
Without loss of generality, we can identify $\mathbf{V}$ with 
$\CC^2$,
with the two forms defined by
\begin{align*}
\langle u,v\rangle &= u'v' + u''v'', \\
[u,v] &= u' v'' - u'' v', 
\end{align*}
for $u=\left[\begin{smallmatrix}u'\\ u''\end{smallmatrix}\right]$, $v=\left[\begin{smallmatrix}v'\\ v''\end{smallmatrix}\right]$. 
Let $\AAA$ be the corresponding affine space (the complex plane). 
Each pair of points $A,B\in\AAA$ gives rise to a vector $\overrightarrow{AB}$ 
that moves $A$ to~$B$. 

\begin{definition}
\label{def:x-and-S}
For $A,B,C\in\AAA$, we define 
\begin{alignat}{23}
\label{eq:x(A,B)}
x(A,B)&=\langle \overrightarrow{AB} , \overrightarrow{AB}\rangle  &&\quad \text{(``squared distance between $A$ and $B$''),} \\
\label{eq:S(A,B,C)}
S(A,B,C) &= 2\,[\overrightarrow{AB},\overrightarrow{AC}] &&\quad \text{(``$4\,\times\!$ signed area of the triangle $ABC$'')}. 
\end{alignat}
\end{definition}


\begin{prop}[{\textrm{Heron's formula}}]
For any triple of points $A,B,C\in\AAA$, the ``measurements'' 
$x(A,B)$, $x(A,C)$, $x(B,C)$ and $S(A,B,C)$ satisfy 
\[
(S(A,B,C))^2=H(x(A,B), x(A,C), x(B,C))
\]
where we use the notation
\begin{equation}
\label{eq:Hpqr}
H(p,q,r)= -p^2-q^2-r^2+2pq+2pr+2qr. 
\end{equation}
\end{prop}

There is also a ``converse Heron theorem"  (Lemma~\ref{lem:recover-triangle} below). To state it properly, we need to introduce the group $\textup{Aut}(\AAA)$ of orientation-preserving isometries of~$\AAA$.

\begin{lem}
\label{lem:recover-triangle}
Given complex numbers $p,q,r,s$ satisfying $s^2=H(p,q,r)$, 
at least one of them nonzero, 
there exists a triangle $ABC$ in~$\AAA$ such that $x(A,B)=p$, $x(A,C)=q$, $x(B,C)=r$, and $S(A,B,C)=s$. 
Moreover such a triangle is unique up to the action of $\textup{Aut}(\AAA)$. 
\end{lem}

\begin{proof}
We note that $\textup{Aut}(\AAA) = \textup{SO}(\VVV) \ltimes T(\VVV)$, where $T(\VVV)$ is the group of translations by an element of~$\VVV$. 
Since $\textup{SO}(\VVV)$ acts freely and transitively on the unit sphere in $\VVV$, the claim will follow from Lemma~\ref{lem:recover-triangle-1} below.
\phantom\qedhere
\end{proof}

\begin{lem}
\label{lem:recover-triangle-1}
Given $A,B\in\AAA$ with $x(A,B)=p\neq 0$, and three numbers $q,r,s\in\CC$ \linebreak[3] satisfying $s^2=H(p,q,r)$, there exists a unique $C\in\AAA$ such that $x(A,C)=q$, $x(B,C)=r$, and $S(A,B,C)=s$. 
\end{lem}

\begin{proof}
Let $u=\overrightarrow{AB}=\left[\begin{smallmatrix}u'\\ u''\end{smallmatrix}\right]$.
We want to find a vector $v=\overrightarrow{AC}=\left[\begin{smallmatrix}v'\\ v''\end{smallmatrix}\right]$ satisfying 
\begin{align}
\label{eq:recover-triangle1}
(v')^2 + (v'')^2 &= q,\\
\label{eq:recover-triangle2}
(u'-v')^2 + (u''-v'')^2 &= r,\\
\label{eq:recover-triangle3}
2(u'v''-u''v') &= s.
\end{align}
Subtracting \eqref{eq:recover-triangle2} from~\eqref{eq:recover-triangle1} gives a linear equation in the unknowns $v'$ and~$v''$. Together with \eqref{eq:recover-triangle3}, this yields 
$v'=\frac{u'(p+q-r)-u''s}{2p}$ and $v''=\frac{u''(p+q-r)+u's}{2p}$. 
It is straightforward to check that we get a solution to
\eqref{eq:recover-triangle1}--\eqref{eq:recover-triangle3}. 
\end{proof}

\begin{defn}
\label{defn:polygon}
A labeled \emph{polygon} (specifically an $n$-gon) in~$\AAA$ is an ordered $n$-tuple of \emph{vertices} $P=(A_1,\dots,A_n)\in\AAA^n$,
with $n\ge 3$. 
Such a polygon gives rise to the \emph{measurements}
\begin{alignat}{3}
\label{eq:xij(P)}
x_{ij}=x_{ij}(P)&=x(A_i,A_j), 
\\
\label{eq:Sijk(P)}
S_{ijk}=S_{ijk}(P)&=S(A_i,A_j,A_k), 
\end{alignat}
for all distinct $i,j,k\in\{1,\dots,n\}$. 
We denote by 
\begin{equation}
\label{eq:z(P)}
\xS(P)=(x_{ij})\sqcup (S_{ijk})
\end{equation}
the labeled collection of all these measurements. 
This collection of numbers (or, depending on the point of view, functions on the configuration space~$\AAA^n$) satisfies many identities, including the obvious symmetries 
\begin{align*}
x_{ij}&=x_{ji} \\
S_{ijk}&=-S_{ikj}=-S_{jik}=S_{jki}=S_{kij}=-S_{kji}
\end{align*}
and the Heron equations 
\begin{equation}
\label{eq:Heron-ijk-1}
S_{ijk}^2 =H(x_{ij}, x_{jk}, x_{ik})
\end{equation}
(cf.\ \eqref{eq:Hpqr}). 
The full list of relations satisfied by the $x_{ij}$'s and $S_{ijk}$'s 
is given by the ``second fundamental theorem'' of invariant theory for the special orthogonal group~$\operatorname{SO}(2,\CC)$, see \cite[Section~II.17]{weyl1946}. 
\end{defn}

\begin{defn}
\label{defn:triangulation-graph}
A \emph{triangulated cycle} (or simply a \emph{triangulation}, when the context allows) 
is a particular kind of unoriented simple graph~$G$ on $n$ vertices $1,\dots,n$. Such a graph must have $2n-3$ edges: $n$~\emph{sides} $\{1,2\}, \{2,3\},\dots,\{n-1,n\},\{1,n\}$ 
forming a distinguished $n$-cycle, 
together with $n-3$ non-side edges called \emph{diagonals}. The key requirement is that $G$ must possess a planar realization of the following kind: take a convex $n$-gon on the real plane with vertices cyclically labeled $1,\dots,n$, triangulate it by diagonals, and consider the resulting~graph. 

More generally, in what follows, any pair $\{i,j\}$ of non-adjacent distinct vertices on the distinguished $n$-cycle will be called a \emph{diagonal}.

We note that each diagonal in a triangulation~$G$ belongs to exactly two \emph{triangles}, i.e., $K_3$-subgraphs of~$G$. 
\end{defn}

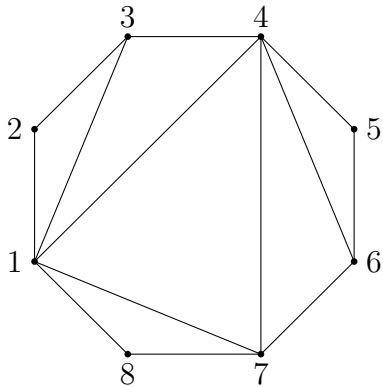
\begin{figure}[ht]
\begin{center}
\setlength{\unitlength}{2.5pt}
\begin{picture}(50,50)(0,-2)
\multiput(14,0)(0,48){2}{\line(1,0){20}}
\multiput(0,14)(48,0){2}{\line(0,1){20}}
\multiput(0,14)(34,34){2}{\line(1,-1){14}}
\multiput(0,34)(34,-34){2}{\line(1,1){14}}
\put(0,14){\line(14,34){14}}
\put(0,14){\line(1,1){34}}
\put(0,14){\line(34,-14){34}}
\put(34,0){\line(0,1){48}}
\put(48,14){\line(-14,34){14}}

\multiput(14,0)(0,48){2}{\circle*{1}}
\multiput(34,0)(0,48){2}{\circle*{1}}
\multiput(0,14)(48,0){2}{\circle*{1}}
\multiput(0,34)(48,0){2}{\circle*{1}}

\put(-3,14){\makebox(0,0){$1$}}
\put(-3,34){\makebox(0,0){$2$}}
\put(51,14){\makebox(0,0){$6$}}
\put(51,34){\makebox(0,0){$5$}}
\put(14,-3){\makebox(0,0){$8$}}
\put(34,-3){\makebox(0,0){$7$}}
\put(14,51){\makebox(0,0){$3$}}
\put(34,51){\makebox(0,0){$4$}}

\end{picture}
\end{center}
\caption{A triangulated $n$-cycle, $n=8$.}
\label{fig:triangulated-octagon}
\vspace{-.2in}
\end{figure}

\begin{defn}
\label{defn:triangulated-polygon}
A~\emph{triangulated polygon} $T=(P,G)$ is a polygon~$P$ as above together with a specific choice of a triangulation~$G$ as in Definition~\ref{defn:triangulation-graph}. 
Once this choice has been made, it makes sense to consider the labeled subcollection of measurements 
\[
\xS_G(P)=(x_{ij})\sqcup (S_{ijk}), 
\]
which only includes the values $x_{ij}$ corresponding to the sides and diagonals of~$T$
(in other words, the edges $\{i,j\}$ of the graph~$G$), and the signed areas $S_{ijk}$ corresponding to the triangles of the triangulation~$G$.  
\end{defn}

\begin{example}
\label{example:quadrilateral}
The simplest nontrivial case is $n=4$.
A quadrilateral $(A_1,A_2,A_3,A_4)$ has two triangulations, involving diagonals $A_1A_3$ and $A_2A_4$, respectively. 
Figure~\ref{fig:quadrilateral-ijkl-abcdef-pqrs} shows these two triangulations, along with their respective measurement data, which involve the measurements
\begin{alignat}{11}
\label{eq:abcd-ABCD}
a&=x_{14}\,,\quad & b&=x_{12}\,,\quad & c&=x_{23}\,,\quad &d&=x_{34}\,, \quad
&e&=x_{13}\,,\quad & f&=x_{24}\,,\quad \\
\label{eq:pqrs-ABCD}
 p&=S_{123}\,,\quad& q&=S_{134}\,,\quad & r&=S_{124}\,,\quad &s&=S_{234}\,.
\end{alignat}
\end{example}

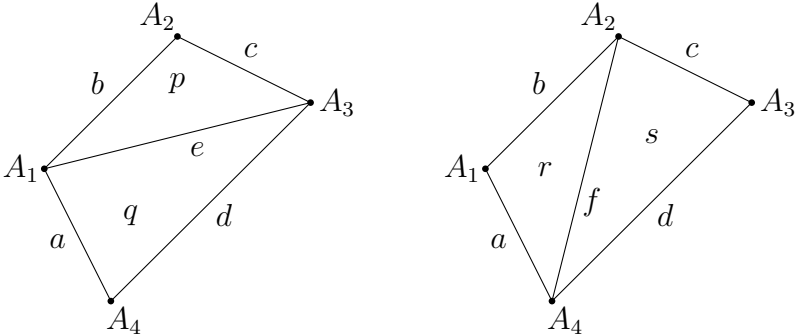
\begin{figure}[ht]
\begin{center}
\setlength{\unitlength}{2.5pt}
\begin{picture}(60,45)(0,9)
\put(10,30){\line(1,-2){10}}
\put(10,30){\line(1,1){20}}
\put(10,30){\line(4,1){40}}
\put(50,40){\line(-2,1){20}}
\put(50,40){\line(-1,-1){30}}
\put(10,30){\circle*{1}}
\put(50,40){\circle*{1}}
\put(20,10){\circle*{1}}
\put(30,50){\circle*{1}}
\put(6.5,30){\makebox(0,0){$A_1$}}
\put(27,53){\makebox(0,0){$A_2$}}
\put(54,40){\makebox(0,0){$A_3$}}
\put(22,7){\makebox(0,0){$A_4$}}
\put(18,43){\makebox(0,0){$b$}}
\put(12,19){\makebox(0,0){$a$}}
\put(37,23){\makebox(0,0){$d$}}
\put(33,33){\makebox(0,0){$e$}}
\put(41,48){\makebox(0,0){$c$}}
\put(30,43){\makebox(0,0){$p$}}
\put(23,23){\makebox(0,0){$q$}}
\end{picture}
\quad
\begin{picture}(60,45)(0,9)
\put(10,30){\line(1,-2){10}}
\put(10,30){\line(1,1){20}}
\put(50,40){\line(-2,1){20}}
\put(50,40){\line(-1,-1){30}}
\put(20,10){\line(1,4){10}}
\put(10,30){\circle*{1}}
\put(50,40){\circle*{1}}
\put(20,10){\circle*{1}}
\put(30,50){\circle*{1}}
\put(6.5,30){\makebox(0,0){$A_1$}}
\put(27,53){\makebox(0,0){$A_2$}}
\put(54,40){\makebox(0,0){$A_3$}}
\put(22,7){\makebox(0,0){$A_4$}}
\put(18,43){\makebox(0,0){$b$}}
\put(12,19){\makebox(0,0){$a$}}
\put(37,23){\makebox(0,0){$d$}}
\put(41,48){\makebox(0,0){$c$}}
\put(26,25){\makebox(0,0){$f$}}
\put(19,30){\makebox(0,0){$r$}}
\put(35,35){\makebox(0,0){$s$}}
\end{picture}
\end{center}
\caption{Measurement data for two triangulations of a plane quadrilateral.}
\label{fig:quadrilateral-ijkl-abcdef-pqrs}
\vspace{-.2in}
\end{figure}

\begin{cor}
\label{cor:triangulation-recovery}
Let $G$ be a triangulated $n$-cycle,
cf.\ Definition~\ref{defn:triangulation-graph}. 
Let 
\[
\xS=(x_{ij})\sqcup (S_{ijk})
\] 
be a collection of complex numbers labeled by the edges $\{i,j\}$ and the oriented triangles $(i,j,k)$ of~$G$. 
Assume that Heron's equation~\eqref{eq:Heron-ijk-1} holds 
for each triangle $(i,j,k)$ in~$G$, and furthermore $x_{ij}\neq0$ for each diagonal $\{i,j\}$ in~$G$. 
Then there exists an $n$-gon~$P$ with $\xS_G(P)=\xS$. 
Moreover $P$ is unique up to the action of $\textup{Aut}(\AAA)$.  In~particular, all the measurements in $\xS(P)$ are uniquely determined~by~$\xS_G(P)$. 
\end{cor}

\begin{proof}
This follows by repeated application of Lemma~\ref{lem:recover-triangle}/Lemma~\ref{lem:recover-triangle-1}. 
\end{proof}

By Corollary~\ref{cor:triangulation-recovery}, a polygon can be uniquely recovered from the measurement data 
associated with an arbitrary triangulation (as long the diagonal lengths are nonzero). 
In particular, the measurement data coming from two different triangulations uniquely determine each other. 
It is natural to ask for an explicit description of the corresponding transition maps. Since any two triangulations can be connected by a sequence of \emph{flips} (cf.\  Definition~\ref{def:flipping} below), it suffices to understand the case of a quadrilateral. 

With notation \eqref{eq:abcd-ABCD}--\eqref{eq:pqrs-ABCD}, 
Corollary~\ref{cor:triangulation-recovery} (for $n=4$) asserts that the measurements $(a,b,c,d,e,p,q)$ determine $(a,b,c,d,f,r,s)$, and vice versa, provided $e\neq 0$ and $f\neq 0$. 
The next proposition describes this correspondence explicitly. 

\pagebreak[3]

\begin{prop}
\label{prop:identities-ABCD}
Let $(A_1,A_2,A_3,A_4)$ be a $4$-gon in~$\AAA$. 
Denote the associated $10$ measurements by $a,b,c,d,e,f,p,q,r,s$, 
as shown in \eqref{eq:abcd-ABCD}--\eqref{eq:pqrs-ABCD} and  Figure~\ref{fig:quadrilateral-ijkl-abcdef-pqrs}.~Then 
\begin{align}
\label{eq:heron-p}
p^2 &= H(b,c,e),\\
\label{eq:heron-q}
q^2 &= H(a,d,e),\\
\label{eq:heron-r}
r^2 &= H(a,f,b),\\
\label{eq:heron-s}
s^2 &= H(c,f,d),\\
\label{eq:additivity}
r+s&=p+q,\\
\label{eq:bretschneider}
4ef &= (p+q)^2 + (a-b+c-d)^2,\\
\label{eq:bilinear}
e(r-s) &= p(a-d) + q(b-c). 
\end{align}
\end{prop}

\begin{proof}
Each of these identities 
can be verified by expressing the involved quantities in terms of the coordinates of the relevant points on the plane. 
Equations~\eqref{eq:heron-p}--\eqref{eq:heron-s} are instances of Heron's formula. 
Equation~\eqref{eq:additivity} reflects the fact that the signed area of a quadrilateral 
can be obtained by cutting it into two triangles by either of the two diagonals,
and adding their areas. 
Equation~\eqref{eq:bretschneider} is known as \emph{Bretschneider's formula} 
for the (squared) area of a quadrilateral. 
Modulo~\eqref{eq:additivity}, equation~\eqref{eq:bilinear} can be interpreted as the $\operatorname{SO}(2)$ instance of 
\cite[Section~II.17, relations~$J_3$]{weyl1946}.
\end{proof}

Motivated by Proposition~\ref{prop:identities-ABCD}, we introduce the following notion. 

\begin{definition}
\label{defn:heronian-diamond}
A \emph{Heronian diamond} is an ordered 10-tuple of complex numbers 
$(a,b,c,d,e,f,p,q,r,s)$
satisfying equations \eqref{eq:heron-p}--\eqref{eq:bilinear}.
Instead of listing the components of a Heronian diamond
as a row of 10~numbers, we will typically arrange them 
in a diamond pattern, as shown in 
Figure~\ref{fig:heronian-diamond}. 
\end{definition}

\begin{rem}
\label{rem:heronian-diamond}
Proposition~\ref{prop:identities-ABCD} can be restated as saying that for any quadrilateral on the plane~$\AAA$, the associated 10~measurements 
($6$~squared distances 
and $4$~signed areas), when properly arranged, will form a Heronian diamond. 
\end{rem}

\begin{figure}[ht]
\includegraphics[scale=0.7]{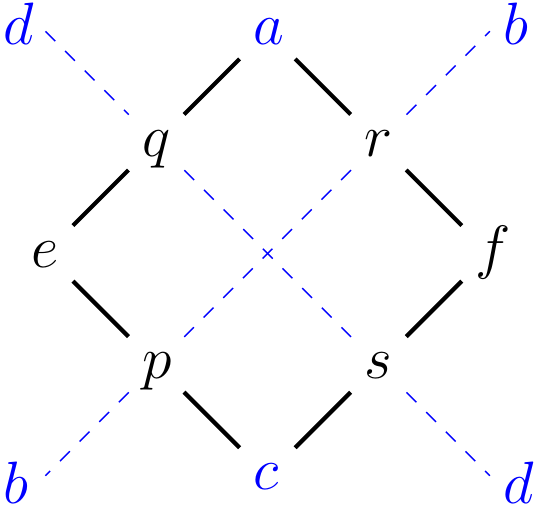}
\caption{A Heronian diamond. 
Here $b$ and $d$ are associated to the dashed lines extending the bimedians of the diamond.  The remaining 8 numbers are placed at the vertices of the diamond and at the midpoints of its sides.
}
\vspace{-.2in}
\label{fig:heronian-diamond}
\end{figure}

\begin{rem}
\label{rem:heronian-redundancy}
Some of the seven conditions \eqref{eq:heron-p}--\eqref{eq:bilinear} appearing in the definition of a Heronian diamond 
are redundant: it is easy to check that equations \eqref{eq:heron-p}--\eqref{eq:heron-q} (or equations \eqref{eq:heron-r}--\eqref{eq:heron-s}) follow from the remaining five. It is however convenient to work with all these seven conditions, for the sake of symmetry (cf.\ Proposition~\ref{prop:diamond-reflect} below) as well as conceptual clarity. 
\end{rem}

\begin{prop}
\label{prop:heron-propagates}
Let $(a,b,c,d,e,p,q)$ be a 7-tuple of complex numbers  satisfying equations~\eqref{eq:heron-p}--\eqref{eq:heron-q}.
Assume that $e \neq 0$. Then there exist unique $f,r,s\in\CC$ such that $(a,b,c,d,e,f,p,q,r,s)$ is a Heronian diamond.
Specifically, 
\begin{align}
\label{eq:f=}
f&=\frac{(p+q)^2 + (a-b+c-d)^2}{4e} ,\\[.1in]
\label{eq:r=}
r&=\frac{p(e+a-d)+q(e-c+b)}{2e},\\[.1in]
\label{eq:s=}
s&=\frac{p(e-a+d)+q(e+c-b)}{2e}. 
\end{align}
\end{prop}

\begin{proof}
We get \eqref{eq:f=} from~\eqref{eq:bretschneider}, 
and \eqref{eq:r=}--\eqref{eq:s=} from \eqref{eq:additivity} and~\eqref{eq:bilinear}. 
One then checks that  \eqref{eq:heron-r}--\eqref{eq:heron-s} are satisfied, cf.\ Remark~\ref{rem:heronian-redundancy}.
\end{proof}

Reflecting a Heronian diamond in a horizontal or vertical axis of symmetry produces a Heronian diamond. 
More precisely: 

\begin{prop}
\label{prop:diamond-reflect}
Let $(a,b,c,d,e,f,p,q,r,s)$ be a Heronian diamond. 
Then 
\begin{itemize}
\item
$(c,d,a,b,e,f,q,p,s,r)$ is a Heronian diamond;
\item
if $e\ne 0$, then $(a,d,c,b,f,e,s,r,q,p)$ is a Heronian diamond. 
\end{itemize}
\end{prop}

\begin{proof}
The first statement is easy: as a result of the interchanges $a\leftrightarrow c$, $b\leftrightarrow d$, $p\leftrightarrow q$, and $r\leftrightarrow s$, the identities \eqref{eq:heron-p}--\eqref{eq:bilinear} get permuted among themselves. 

The second reflection, across a vertical line, interchanges $d\leftrightarrow b$, $q\leftrightarrow r$, $e\leftrightarrow f$, and $p\leftrightarrow s$. 
Again, the identities \eqref{eq:heron-p}--\eqref{eq:bretschneider} get permuted---but \eqref{eq:bilinear} is replaced~by
\begin{equation}
\label{eq:f(p-q)}
f(p-q) = r(c-d) + s(b-a). 
\end{equation}
Thus, we need to deduce \eqref{eq:f(p-q)} from~\eqref{eq:heron-p}--\eqref{eq:bilinear}. 
It will be convenient to denote
$g=a - b + c - d$.
We then obtain: 
\begin{align*}
    &4ef(p-q) \\
\text{by \eqref{eq:bretschneider}}\    =&((p+q)^2+g^2)(p-q)\\
   =&(p^2-q^2)(p+q)+g^2(p-q) \\
\text{by \eqref{eq:heron-p},\eqref{eq:heron-q}}\   =&(H(b,c,e) - H(a,d,e))(p+q)+g^2(p-q) \\
=&((a+b-c-d)g + 2e(-a+b+c-d))(p+q)+g^2(p-q) \\
=&g((a\!+\!b\!-\!c\!-\!d)(p\!+\!q)\!+\!g(p-q)) \!+\!2e(-a\!+\!b\!+\!c\!-\!d)(p\!+\!q) \\
\text{by \eqref{eq:additivity}}\  =&g(2p(a-d)+2q(b-c))+2e(r+s)(-a+b+c-d) \\
\text{by \eqref{eq:bilinear}}\  =& 2e(r-s)(a-b+c-d)+2e(r+s)(-a+b+c-d) \\
=&4er(c-d)+4es(b-a). 
\end{align*}
Dividing by $4e$ (here we use that $e\neq 0$), we get~\eqref{eq:f(p-q)}. 
\end{proof}

\begin{cor}
\label{cor:heronian-propagate}
In a Heronian diamond $(a,b,c,d,e,f,p,q,r,s)$, once the components $a,b,c,d$ (shown in blue in Figure~\ref{fig:heronian-diamond}) have been fixed, the values $e,p,q$ determine $f,r,s$ uniquely (provided $e\neq 0$), and vice versa (provided $f\neq 0$). 
\end{cor}

\begin{proof}
Combine Propositions~\ref{prop:heron-propagates} and~\ref{prop:diamond-reflect}.
\end{proof}

The next two lemmas will be needed in Section~\ref{sec:heron}. 

\begin{lem}[{\rm Heronian diamonds with $a=q=r=0$}] 
\label{lem:diamond-0-top}
Complex numbers 
\[
0, b, c, d, e, f, p, 0, 0, s
\]
form a Heronian diamond if and only if
\begin{align}
\label{eq:heron-p-again}
p^2 &= H(b,c,e),\\
\label{eq:d=e}
d&=e, \\
\label{eq:f=b}
f&=b, \\
\label{eq:s=p}
s&=p. 
\end{align}
\end{lem}

\begin{proof}
Under the assumptions $a=q=r=0$, we have 
\[
\eqref{eq:heron-p}\!\!-\!\!\eqref{eq:bilinear} 
\Longleftrightarrow
\begin{cases}
&\hspace{-3mm}p^2= H(b,c,e)\\
&\hspace{-3mm}0=H(0,d,e)=-(d-e)^2\\
&\hspace{-3mm}0=H(0,f,b)=-(f-b)^2\\
&\hspace{-3mm}s^2=H(c,f,d)\\
&\hspace{-3mm}s=p\\
&\hspace{-3mm}4ef=p^2+(-b+c-d)^2\\
&\hspace{-3mm}es=dp
\end{cases}
\Longleftrightarrow
\begin{cases}
&\hspace{-3mm}p^2= H(b,c,e)\\
&\hspace{-3mm}d=e\\
&\hspace{-3mm}f=b\\
&\hspace{-3mm}s^2=H(c,b,e)\\
&\hspace{-3mm}s=p\\
&\hspace{-3mm}p^2=4eb-(-b+c-e)^2.
\end{cases}
\]
Since $4eb-(-b+c-e)^2=H(b,c,e)$, the claim follows.
\end{proof}

\begin{lem}[{\rm Heronian diamonds with $c=p=s=0$}] 
\label{lem:diamond-0-bottom}
Complex numbers 
\[
a, b, 0, d, e, f, 0, q, r, 0
\]
form a Heronian diamond if and only if
\begin{align}
\label{eq:heron-q-again}
q^2 &= H(a,d,e),\\
\label{eq:b=e}
b&=e, \\
\label{eq:f=d}
f&=d, \\
\label{eq:r=q}
r&=q. 
\end{align}
\end{lem}

\begin{proof}
The proof is completely analogous to the proof of Lemma~\ref{lem:diamond-0-top}. 
Alternatively, combine Lemma~\ref{lem:diamond-0-top} with Proposition~\ref{prop:diamond-reflect}. 
\end{proof}

\begin{cor}
\label{cor:heronian-propagate-near-boundary}
In a Heronian diamond $(a,b,c,d,e,f,p,q,r,s)$ with $a=q=r=0$, the values $e,b,p$ determine $f,d,s$ uniquely, and vice versa. 
In a Heronian diamond  $(a,b,c,d,e,f,p,q,r,s)$ with $c=p=s=0$, the values $d,e,q$ determine $b,f,r$ uniquely, and vice versa. 
\end{cor}

\pagebreak[3]

\section{Heronian friezes}
\label{sec:heron}

Remark~\ref{rem:heronian-diamond} 
implies the following statement. 

\begin{prop}
\label{prop:identities-ijkl}
Let $P\!=\!(A_1,\dots,A_n)$ be a polygon in~$\AAA$. 
For any four vertices $A_i,A_j,A_k,A_\ell$ of~$P$, the corresponding 10~measurements form a Heronian diamond shown in Figure~\ref{fig:heronian-diamond-xS}. 
More explicitly, the measurements in 
$\xS(P)\!=\!(x_{ij})\sqcup (S_{ijk})$
(cf.\ Definition~\ref{defn:polygon}) satisfy the following identities,  
for any distinct $i,j,k,\ell\in\{1,\dots,n\}$:
\begin{align}
\label{eq:Heron-ijk-2}
S_{ijk}^2 &= H(x_{ij}, x_{jk}, x_{ik}) ,\\
\label{eq:split-ijkl}
S_{ijk} + S_{ik\ell} &= S_{ij\ell} + S_{jk\ell} ,\\
\label{eq:Bretschneider-ijk}
4x_{ik}x_{j\ell} &= (S_{ijk}+S_{ik\ell})^2 + (x_{ij}-x_{jk}+x_{k\ell}-x_{i\ell})^2 ,\\
\label{eq:bilinear-ijk}
x_{ik}(S_{ij\ell} - S_{jk\ell}) &= S_{ijk}(x_{i\ell} - x_{k\ell}) + S_{ik\ell}(x_{ij}-x_{jk}). 
\end{align}
\end{prop}

\begin{figure}[ht]
\vspace{-.1in}
\includegraphics[scale=0.75]{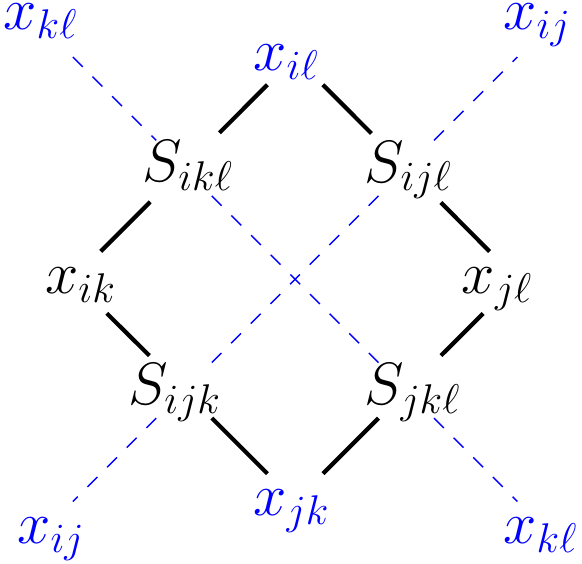}
\caption{A Heronian diamond for a quadruple of vertices with labels $i,j,k,\ell$.
}
\vspace{-.1in}
\label{fig:heronian-diamond-xS}
\end{figure}

Motivated by Figure~\ref{fig:heronian-diamond-xS}, we introduce the notion of a Heronian frieze, cf.\ Definition~\ref{defn:heronian-frieze} below. 
Informally, a Heronian frieze is a collection of numbers arranged in a pattern shown in Figure~\ref{fig:heronian-frieze-1}, and satisfying the Heronian diamond equations for all diamonds in the pattern (plus some additional conditions near the upper and lower boundaries). 
We next proceed to a formal definition. 


\begin{figure}[ht]
\begin{center}
\includegraphics[height=2in]{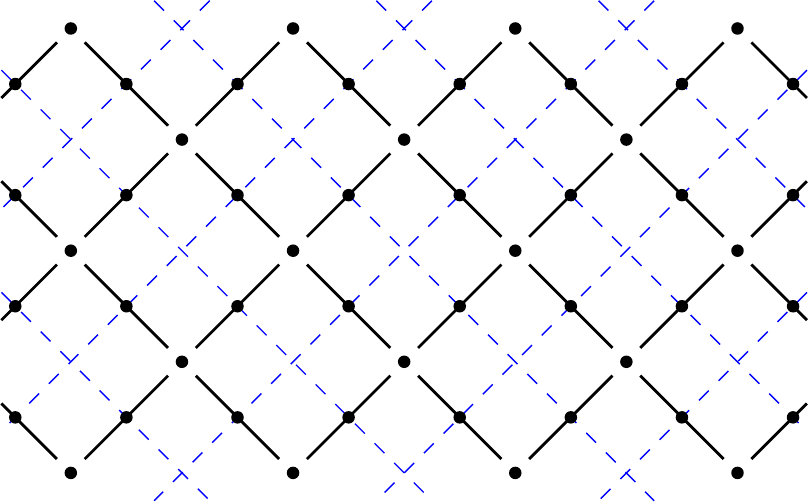}
\end{center}
\caption{The combinatorial pattern underlying a Heronian frieze of order~$n=4$.}
\label{fig:heronian-frieze-1}
\end{figure}

We begin by introducing the relevant indexing sets. 

\begin{defn}
\label{defn:heronian-indexing}
For $n\ge 4$, let $N_n$ and $L_n$ be the sets defined by
\begin{align}
\label{eq:Nn}
N_n&= \bigl\{(i,j) \in (\ZZ\times\tfrac12\ZZ)\cup (\tfrac12\ZZ\times\ZZ)
\colon 0 \leq j-i \leq n\bigr\},\\
\label{eq:Ln}
L_n&= \bigl\{(i+\tfrac12,\neline)\colon i \in \ZZ\bigr\}
   \cup\bigl\{(\seline,j+\tfrac12)\colon j \in\ZZ\bigr\}. 
\end{align}
The (disjoint) union $I_n=N_n\cup L_n$ will serve as the indexing set for the Heronian friezes. 
We visualize this set as follows, see Figure~\ref{fig:frieze-indices}. 
We interpret $\ZZ^2$ as the set of integer points for the coordinate system whose axes are rotated clockwise by $\pi/4$ with respect to the usual placement. 
The~indices in $N_n$ (``the nodes'') are the points $(i,j)$ in the strip \hbox{$0 \leq j-i \leq n$} whose coordinates $i,j$ are half-integers, with at least one of them an integer. 
The indices in $L_n$ (``the lines'') represent straight lines parallel to the coordinate axes, with half-integer offsets. 

We will refer to the indices $(i,j) \in N_n$ with $1 \leq j-i \leq n-1$ as the \emph{interior} nodes of~$N_n$. 
\end{defn}

\begin{figure}[ht]
\begin{center}
\includegraphics[scale=0.8]{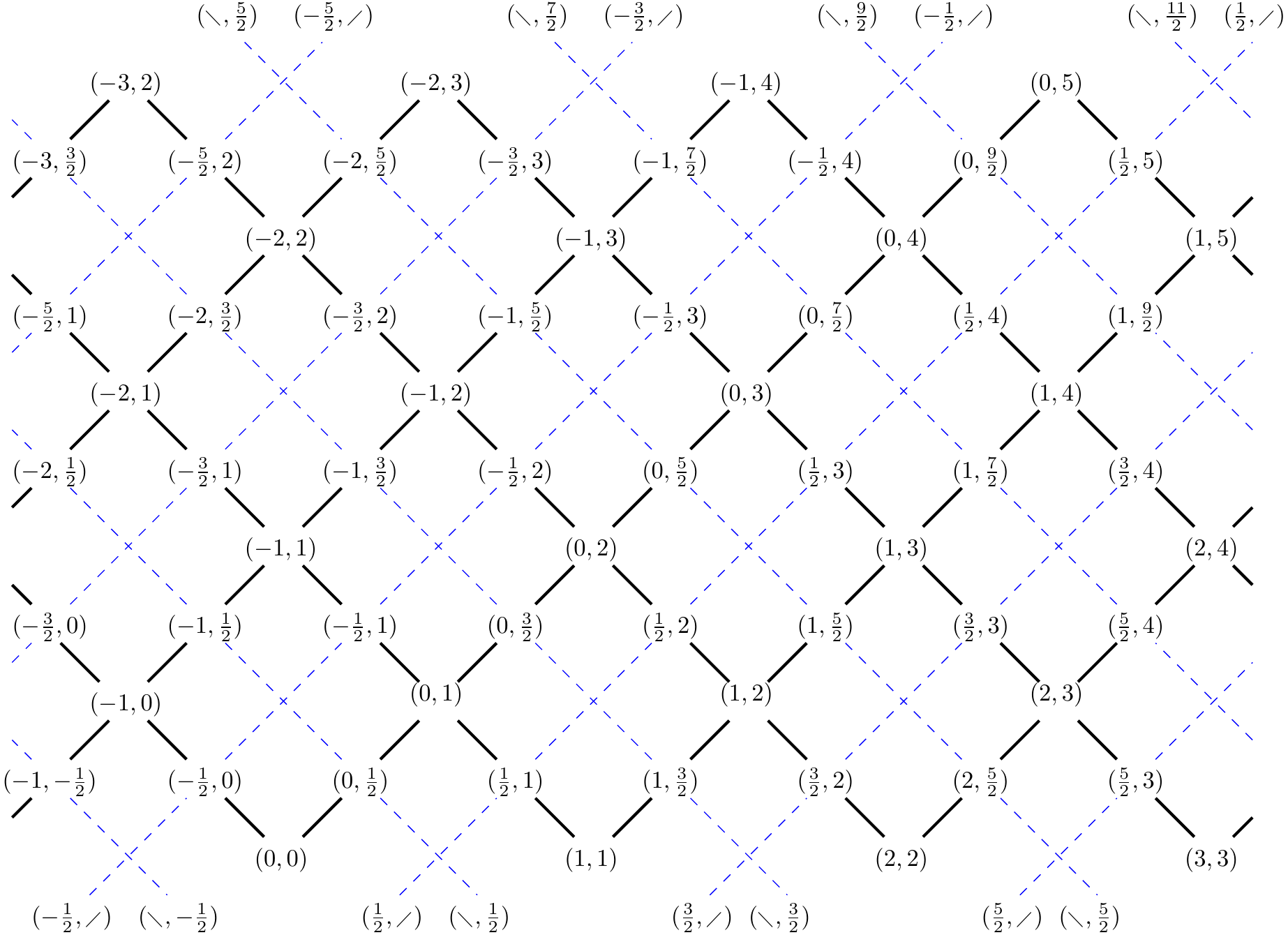}
\end{center}
\caption{The indexing set for a Heronian frieze of order~$n=5$. The indices in  $L_n$ correspond to the dashed lines; see the top and bottom~rows of the picture. 
The remaining 11~rows of indices constitute the set of nodes~$I_n$. The middle 7~rows are the interior nodes.}
\label{fig:frieze-indices}
\end{figure}

\begin{definition}
\label{defn:heronian-frieze}
A \emph{Heronian frieze} of order $n\ge 4$ is an array 
$\mathbf{z} = (z_{\alpha})_{\alpha \in I_n}$
of complex numbers indexed by the set $I_n$ 
(see Definition~\ref{defn:heronian-indexing}) 
which satisfies the following local conditions. 
The main condition is that for every 10-tuple of indices shown in Figure~\ref{fig:heronian-frieze-blowup} (with $(i,j)\in N_n\cap\ZZ^2$ an interior node), we require the corresponding 10~entries 
\[
(z_{(i,j+1)}, z_{(i+\frac12,\smallneline)}, z_{(i+1,j)}, z_{(\smallseline,j+\frac12)}, z_{(i,j)}, z_{(i+1,j+1)}, z_{(i+\frac12,j)}, z_{(i,j+\frac12)} ,z_{(i+\frac12,j+1)}, z_{(i+1,j+\frac12)})
\]
to form a Heronian diamond. (For a dictionary between this notation and the notation in Definition~\ref{defn:heronian-diamond}, compare Figures~\ref{fig:heronian-diamond} and~\ref{fig:heronian-frieze-blowup}.)
In addition, we impose the boundary conditions
\begin{align}
\label{eq:border-equalities}
z_{(i,i)} = z_{(i,i+n)} = z_{(i,i+\frac12)} = z_{(i,i+n-\frac12)} = 0 \quad (i\in\ZZ). 
\end{align}
\end{definition}

\begin{figure}[ht]
\begin{center}
\includegraphics[height=2.1in]{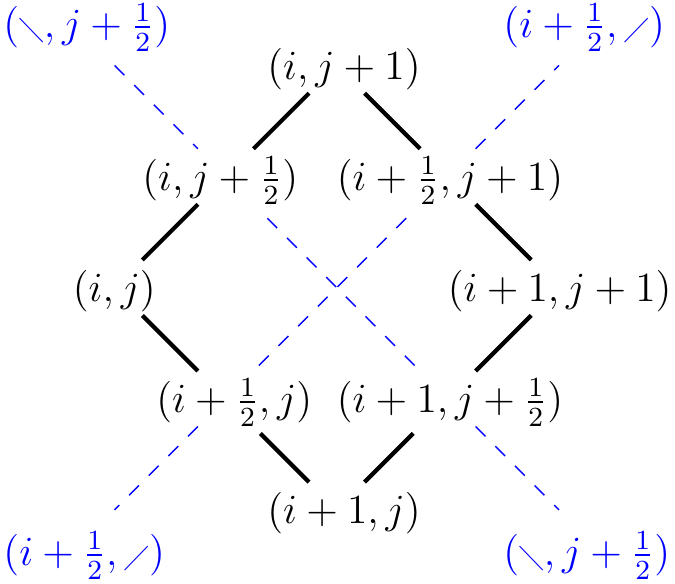}
\vspace{-.1in}
\end{center}
\caption{Indexing set for a diamond in a Heronian frieze. Here $1\leq j-i \leq n-1$.}
\label{fig:heronian-frieze-blowup}
\end{figure}

The notion of a Heronian frieze simplifies under the assumption that all entries indexed by the elements of the set $L_n$ (see~\eqref{eq:Ln}) are equal to each other. (This assumption mirrors the analogous condition traditionally imposed on the classical Coxeter friezes.) We next present the self-contained version of  Definition~\ref{defn:heronian-frieze} in this restricted generality. 

\begin{defn} 
\label{defn:equilateral-heronian-frieze}
Let $b$ be a nonzero complex number. A Heronian frieze of order~$n$ is called \emph{equilateral}  (with the lateral parameter~$b$) if $z_{(i+\frac12,\smallneline)}\!=\!z_{(\smallseline,i+\frac12)}\!=\!b$ for all~$i$. 
Such a frieze can be thought of as an array 
$\mathbf{z} = (z_{(i,j)})_{(i,j) \in N_n}$
of complex \hbox{numbers} indexed by the set $N_n$ (see~\eqref{eq:Nn}) and satisfying the boundary conditions~\eqref{eq:border-equalities} together with the 
following relations, which hold for every node $(i,j)\in\ZZ^2$ with $1\le j-i\le n-1$: 
\begin{align*}
z_{(i+\frac12,j)}^2&=H(b,z_{(i,j)},z_{(i+1,j)}); \\
z_{(i,j+\frac12)}^2&=H(b,z_{(i,j)},z_{(i,j+1)}); \\
z_{(i+\frac12,j)}+z_{(i,j+\frac12)}&=z_{(i+\frac12,j+1)}+z_{(i+1,j+\frac12)};\\
4z_{(i,j)}z_{(i+1,j+1)} &=(z_{(i+\frac12,j)}+z_{(i,j+\frac12)})^2
  +(z_{(i+1,j)}+z_{(i,j+1)}-2b)^2;\\
z_{(i,j)} (z_{(i+\frac12,j+1)}-z_{(i+1,j+\frac12)})&=
   z_{(i+\frac12,j)}(z_{(i,j+1)}-b) +z_{(i,j+\frac12)} (b-z_{(i+1,j)}). 
\end{align*}
An example of an equilateral Heronian frieze (with $b=1$) is shown in Figure~\ref{fig:heronian-frieze-example}. 
\end{defn}

\begin{figure}[ht]
\begin{center}
\includegraphics[scale=0.8]{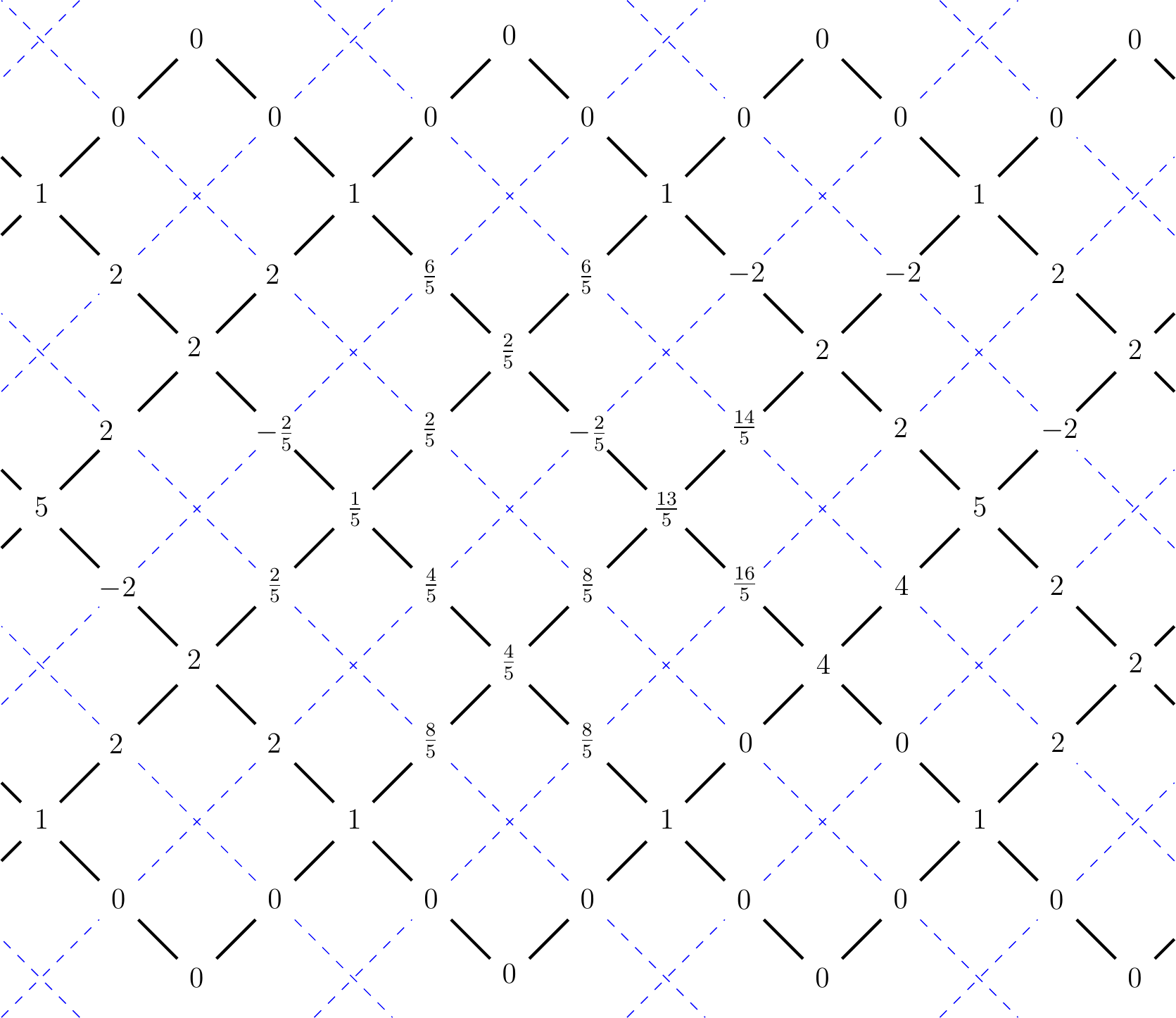}
\vspace{-.1in}
\end{center}
\caption{A fragment of an (equilateral) Heronian frieze of order~6.
The entries associated with the dashed lines (i.e., the ones indexed by the elements of~$L_n$) are all equal~to~1. 
}
\vspace{-.2in}
\label{fig:heronian-frieze-example}
\end{figure}

\pagebreak[3]

The boundary conditions \eqref{eq:border-equalities} imply the following identities. 

\begin{prop}
Let $\mathbf{z} = (z_\alpha)_{\alpha \in I_n}$ be a Heronian frieze of order $n$. Then 
\begin{alignat}{3}
\label{eq:partial-diamond-1}
z_{(i,i+1)} &= z_{(i+\frac12,\smallneline)} = z_{(\smallseline,i+\frac12)}\quad &&(i \in \ZZ),\\
\label{eq:partial-diamond-2}
z_{(i,i+n-1)} &= z_{(i-\frac12,\smallneline)} = z_{(\smallseline,i+n-\frac12)}\quad &&(i \in \ZZ).
\end{alignat}
\end{prop}

\begin{proof}
The diamond condition for the interior node $(i,i+1)$ says that the 10 numbers
\[
z_{(i,i+2)}, z_{(i+\frac12,\smallneline)}, z_{(i+1,i+1)}, z_{(\smallseline,i+\frac32)}, z_{(i,i+1)}, z_{(i+1,i+2)}, z_{(i+\frac12,i+1)}, z_{(i,i+\frac32)} ,z_{(i+\frac12,i+2)}, z_{(i+1,i+\frac32)}
\]
form a Heronian diamond. By \eqref{eq:border-equalities}, three of these numbers vanish: $z_{(i+1,i+1)}\!=\!z_{(i+\frac12,i+1)}\!=\!z_{(i+1,i+\frac32)}\!=\!0$.  Hence Lemma~\ref{lem:diamond-0-bottom} applies, and  $z_{(i,i+1)} \!=\! z_{(i+\frac12,\smallneline)}$ by~\eqref{eq:b=e}.

Similarly, the diamond condition for the node $(i-1,i)$ says that the 10 numbers
\[
(z_{(i-1,i+1)}, z_{(i-\frac12,\smallneline)}, z_{(i,i)}, z_{(\smallseline,i+\frac12)}, z_{(i-1,i)}, z_{(i,i+1)}, z_{(i-\frac12,i)}, z_{(i-1,i+\frac12)} ,z_{(i-\frac12,i+1)}, z_{(i,i+\frac12)})
\]
form a Heronian diamond. The three numbers $z_{(i,i)}$, $z_{(i-\frac12,i)}$, and $z_{(i,i+\frac12)}$ are all zero, so Lemma~\ref{lem:diamond-0-bottom} applies.  By~\eqref{eq:f=d}, we get $z_{(i,i+1)} = z_{(i+\frac12,\smallseline)}$, establishing \eqref{eq:partial-diamond-1}.

Equation~\eqref{eq:partial-diamond-2} is proven in a similar way, by applying Lemma~\ref{lem:diamond-0-top} to the Heronian diamonds associated with the interior nodes $(i,i+n-1)$ and~$(i-1,i+n-2)$.
\end{proof}

\begin{defn}
\label{def:z(P)}
In light of Proposition~\ref{prop:identities-ijkl}
(also compare Figures~\ref{fig:heronian-diamond-xS} and~\ref{fig:heronian-frieze-blowup}), any $n$-gon $P$ in the plane~$\AAA$ gives rise to a Heronian frieze $\mathbf{z}=\mathbf{z}(P)$ of order $n$ by setting
\begin{align}
\label{eq:pattern-restriction-start}
z_{(i,j)} &= x_{\langle i\rangle\langle j\rangle},\\
z_{(i+\frac12,j)}&= S_{\langle i\rangle\langle i+1\rangle\langle j\rangle},\\
z_{(i,j+\frac12)}&= S_{\langle i\rangle\langle j\rangle\langle j+1\rangle},\\
z_{(i+\frac12,\smallneline)}&= x_{\langle i \rangle\langle i+1\rangle},\\
\label{eq:pattern-restriction-end}
z_{(\smallseline,j+\frac12)}&= x_{\langle j \rangle\langle j+1\rangle},
\end{align}
where $\langle m \rangle$ denotes the unique integer in $\{1,\dots,n\}$ satisfying $m \equiv \langle m \rangle \pmod n$. 
(Condition \eqref{eq:border-equalities} holds because $x_{ii}=S_{i,i,i+1}=S_{i,i+1,i+1}=0$  for every $i \in \ZZ$.) 
\end{defn}

Any frieze $\mathbf{z}(P)$ coming from a polygon~$P$ is necessarily periodic:
\begin{alignat}{3}
\label{eq:periodicity-from-polygon-1}
z_{(i,j)}&=z_{(i+n,j+n)} \quad &&(i,j\in N_n),\\
\label{eq:periodicity-from-polygon-2}
z_{(i+\frac12,\smallneline)}&=z_{(i+\frac12+n,\smallneline)} \quad &&(i\in\ZZ),\\
\label{eq:periodicity-from-polygon-3}
z_{(\smallseline,j+\frac12)}&=z_{(\smallseline,j+\frac12+n)}\quad &&(j\in\ZZ). 
\end{alignat}
In fact, \eqref{eq:periodicity-from-polygon-1} can be strengthened as follows: $\mathbf{z}(P)$ possesses the \emph{glide symmetry}
\begin{align}
\label{eq:glide-from-polygon}
z_{(i,j)} &= z_{(j,i+n)}\quad (i,j\in N_n),
\end{align} 
which also reflects the symmetries $x_{ij}\!=\!x_{ji}$ and $S_{ijk}\!=\!S_{jki}$ of the measurements. 
(The~same symmetries appear in the Coxeter-Conway theory of frieze patterns~\cite{conway-coxeter, coxeter-friezes}.) 
We will soon provide a partial converse to this phenomenon, cf.\ Theorem~\ref{thm:frieze-to-pattern}. 

Although the definition of Heronian friezes was motivated by geometry, they are purely algebraic objects, merely tables of numbers satisfying some algebraic relations. These relations can be viewed as recurrences: start by picking some initial data, then propagate away by repeatedly applying Corollary~\ref{cor:heronian-propagate} (or Corollary~\ref{cor:heronian-propagate-near-boundary}) for the Heronian diamonds in the pattern. To describe this procedure in precise terms, we will need to specify the sets of indices corresponding to our choices of initial data.

%
%

\begin{defn}
\label{defn:traversing-Heronian}
A \emph{traversing path} $\pi$ for an order~$n$ Heronian frieze is an ordered collection 
\[
\pi=((i_1, j_1), \dots, (i_{2n-3},j_{2n-3}),\ell_1, \dots, \ell_{n-2})
\]
of $3n-5$ indices in $I_n$ such that
\begin{itemize}
\item
$(i_1,j_1), \dots, (i_{2n-3},j_{2n-3})$ 
are interior nodes in~$N_n$; 
\item
$\ell_1, \dots, \ell_{n-2}$ are lines in $L_n$; 
\item
$j_1 - i_1 = 1$;
\item
$j_{2n-3} - i_{2n-3} = n-1$;
\item
$|i_{k+1} - i_k| + |j_{k+1} - j_k| = \tfrac12$, for $k \in \{1, \dots, 2n-4\}$;
\item
if $i_{2k} \in \ZZ+\frac12$, then $\ell_k = (i_{2k},\neline)\in L_n$; 
\item
if $j_{2k} \in \ZZ+\frac12$, then $\ell_k = (\seline, j_{2k})\in L_n$.
\end{itemize}
The following less formal description is perhaps more illuminating. Let us view $N_n$ as the vertex set of a graph, as shown in Figure~\ref{fig:heronian-frieze-1}, but without the dashed lines.~Then: 
\begin{itemize}
\item
$(i_1,j_1), \dots, (i_{2n-3},j_{2n-3})$ are the nodes lying on a shortest path connecting the lower and upper boundaries of the strip of interior nodes; 
\item
$\ell_1, \dots, \ell_{n-2}$ are the dashed lines intersecting this shortest path. 
\end{itemize}
\end{defn}

\begin{example}
\label{example:heronian-traversing-path}
For $n=5$ (cf.\ Figure~\ref{fig:frieze-indices}), a traversing path consists of $3n-5 = 10$ indices. One example of such a path is
\[
\bigl((0,1), (0,\tfrac32), (0,2), (-\tfrac12,2), (-1,2), (-1, \tfrac52), (-1,3), (\seline,\tfrac32), (-\tfrac12,\neline), (\seline,\tfrac52)\bigr).
\]
\end{example}

\begin{rem}
\label{rem:path=triangulation}
For a Heronian frieze $\mathbf{z}(P)$ obtained from a plane $n$-gon~$P$ as in Definition~\ref{def:z(P)}, a traversing path~$\pi$ corresponds to a particular kind of a triangulation, namely one in which every triangle has at least one of its sides lying on the perimeter of~$P$. (Cf.\ Definition~\ref{defn:thin-triangulation} below.) 
Moreover by Corollary~\ref{cor:triangulation-recovery}, a sufficiently generic polygon~$P$ (hence the entire frieze $\mathbf{z}(P)$) can be recovered from the values of the frieze lying along~$\pi$. 
\end{rem}


\begin{cor}
\label{cor:frieze-is-recurrence}
Let $\mathbf{z} = (z_\alpha)_{\alpha \in I_n}$ be a Heronian frieze of order~$n$. Suppose we know that 
\begin{equation}
\label{eq:nonzero-interior}
\text{$z_{(i,j)}\neq 0\,$ for any $\,(i,j)\in \ZZ^2$ such that  $\,2\le j-i\le n-2$. }
\end{equation}
Then the entire frieze can be uniquely reconstructed from its entries lying on a single traversing path~$\pi$.
\end{cor}

\begin{proof}
Repeatedly apply the recurrences underlying Corollary~\ref{cor:heronian-propagate} and Corollary~\ref{cor:heronian-propagate-near-boundary} to all  Heronian diamonds in the frieze, starting with the ones adjacent to $\pi$ and expanding out.
\end{proof}



To be more specific, the recurrences for rightward propagation in a Heronian freeze are \eqref{eq:f=}--\eqref{eq:s=} (inside the frieze), \eqref{eq:d=e}--\eqref{eq:s=p} (near the top boundary) and \eqref{eq:b=e}--\eqref{eq:r=q} (near the bottom). 
For an equilateral frieze with parameter~$b$, we set $d=b$, and do not need to update the line variables $b$ and $d$. 

\begin{rem}
Corollary~\ref{cor:frieze-to-polygon} leaves open the question of \emph{existence} of a Heronian frieze with the given values along a particular traversing path (subject to an appropriate nonvanishing condition). We answer this question in Section~\ref{sec:laurentness}. 
\end{rem}

\begin{cor}
\label{cor:frieze-to-polygon}
Let $\mathbf{z}$ be a Heronian frieze of order~$n$ satisfying the nonvanishing condition~\eqref{eq:nonzero-interior}. 
Then there exists a (unique) $n$-gon such that 
$\mathbf{z}=\mathbf{z}(P)$. 
\end{cor}

\begin{proof}
Pick a traversing path $\pi$ and construct an $n$-gon~$P$ whose frieze $\mathbf{z}(P)$ agrees with $\mathbf{z}$ along~$\pi$, as in Remark~\ref{rem:path=triangulation}. Then apply Corollary~\ref{cor:frieze-is-recurrence}.
\end{proof}

Corollary~\ref{cor:frieze-to-polygon} implies the following purely algebraic statement. 

\begin{thm}
\label{thm:frieze-to-pattern}
Any Heronian frieze satisfying the nonvanishing condition~\eqref{eq:nonzero-interior} possesses the glide symmetry~\eqref{eq:glide-from-polygon}.  
\end{thm}

\begin{example}
Figure~\ref{fig:heronian-frieze-example} shows the fundamental domain for an equilateral frieze with respect to the glide symmetry. 
\end{example}

\pagebreak[3]
\section{Laurent phenomenon for Heronian friezes}
\label{sec:laurentness}

The main result of this section is the following theorem. 

\begin{thm}
\label{thm:laurentness-full}
Let $G$ be a triangulated $n$-cycle. Then every measurement in $\xS(P)$ (viewed as a function on the configuration space of all $n$-gons~$P$) can be expressed as a Laurent polynomial in the measurements in~$\xS_G(P)$. The denominator of this Laurent polynomial is a monomial in the squared lengths of diagonals~of~$G$.
\end{thm}

In algebraic terms, Theorem~\ref{thm:laurentness-full} asserts that each entry in a generic Heronian frieze can be written as a Laurent polynomial in the initial data associated with a choice of a traversing path, see Corollary~\ref{cor:Laurentness-Heronian-frieze} below. 

Later in this section, we prove a slightly stronger---but more technical---version of Theorem~\ref{thm:laurentness-full}, see Theorem~\ref{thm:laurent-thin}. 

\medskip

The proof of Theorem~\ref{thm:laurentness-full} requires some preparation. 

\begin{defn}
Consider the $n$-cycle with vertices $1,2,\dots,n$ (in~this order), $n\ge 4$. 
Let $i,j,k,\ell$ be four distinct vertices on this cycle, with $i < j$ and $k < \ell$. We say that the diagonals $\{i,j\}$ and $\{k,\ell\}$ \emph{cross} if either $i < k < j < \ell$ or $k < i < \ell < j$. (In~particular, no two diagonals incident to the same vertex cross each other.) 
\end{defn}


\begin{defn}
\label{defn:trimmed}
Let $G$ be a triangulated $n$-cycle, see Definition~\ref{defn:triangulation-graph}. We denote by $E(G)$ the set of edges of~$G$, and by $D(G)\subset E(G)$ the set of diagonals of~$G$. 
For a diagonal $\{i,j\}\notin D(G)$, the \emph{trimming} of $G$ with respect to $\{i,j\}$, denoted $\tau(G,i,j)$, is the induced subgraph of~$G$ whose vertex set includes $i$, $j$, and the endpoints of all diagonals in $D(G)$ which cross~$\{i,j\}$. Note that $\tau(G,i,j)$ is itself a triangulated~cycle. If $G=\tau(G,i,j)$, then we say that $G$ is \emph{trimmed} with respect to~$\{i,j\}$. See Figure~\ref{fig:trimming-octagon}. 

Similarly, the trimming of $G$ with respect to a triple $(i,j,k)$, denoted $\tau(G,i,j,k)$, is 
the induced subgraph of~$G$ whose vertex set includes $i$, $j$, $k$, and the endpoints of all diagonals in $D(G)$ which cross at least one of the diagonals $\{i,j\}$, $\{i,k\}$, $\{j,k\}$. Again, $\tau(G,i,j,k)$ is a triangulated cycle. If $G = \tau(G,i,j,k)$, then we say that $G$ is \emph{trimmed} with respect to $(i,j,k)$.
\end{defn}

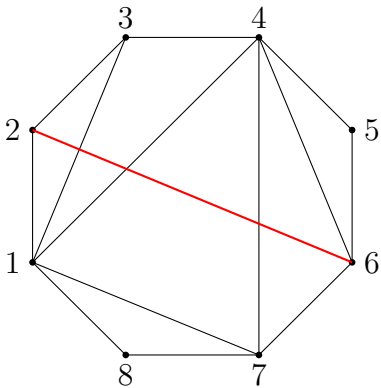
\begin{figure}[ht]
\begin{center}
\setlength{\unitlength}{2.5pt}
\begin{picture}(50,50)(0,-2)
\multiput(14,0)(0,48){2}{\line(1,0){20}}
\multiput(0,14)(48,0){2}{\line(0,1){20}}
\multiput(0,14)(34,34){2}{\line(1,-1){14}}
\multiput(0,34)(34,-34){2}{\line(1,1){14}}
\put(0,14){\line(14,34){14}}
\put(0,14){\line(1,1){34}}
\put(0,14){\line(34,-14){34}}
\put(34,0){\line(0,1){48}}
\put(48,14){\line(-14,34){14}}

\multiput(14,0)(0,48){2}{\circle*{1}}
\multiput(34,0)(0,48){2}{\circle*{1}}
\multiput(0,14)(48,0){2}{\circle*{1}}
\multiput(0,34)(48,0){2}{\circle*{1}}

\put(-3,14){\makebox(0,0){$1$}}
\put(-3,34){\makebox(0,0){$2$}}
\put(51,14){\makebox(0,0){$6$}}
\put(51,34){\makebox(0,0){$5$}}
\put(14,-3){\makebox(0,0){$8$}}
\put(34,-3){\makebox(0,0){$7$}}
\put(14,51){\makebox(0,0){$3$}}
\put(34,51){\makebox(0,0){$4$}}

\thicklines
\put(0,34){\color{red}{\line(48,-20){48}}}
\end{picture}
\end{center}
\caption{The trimming of this triangulated $8$-cycle $G$ with respect to the diagonal $\{2,6\}$ produces a triangulated hexagon $\tau(G,2,6)$ with vertices $1,2,3,4,6,7$.}
\label{fig:trimming-octagon}
\end{figure}

\begin{rem}
\label{rem:can-trim}
When we are interested in recovering a measurement $x_{ij}$ (resp.,~$S_{ijk}$)  of a plane polygon~$P$ from the subset of measurements $\xS_G(P)$ corresponding to a triangulation~$G$, we may always assume, without loss of generality, that $G$ is trimmed with respect to the diagonal~$\{i,j\}$ (resp., the triangle $(i,j,k)$). (Otherwise, we can trim~$G$, and then proceed. The formulas will be exactly the same.) 
\end{rem}

\begin{rem}
\label{rem:trimmed-is-thin}
In a trimmed triangulation $\tau(G,i,j)$, every triangle uses at least one side of the $n$-cycle. Equivalently, no three diagonals form a triangle. 
\end{rem}

\begin{definition}
\label{def:flipping}
Let $e$ be a diagonal in a triangulation~$G$. We denote by $G'=\mu_e(G)$ the unique triangulation (of the same cycle) obtained by replacing $e$ by a different diagonal~$f$. We say that $G'$ is obtained from~$G$ by \emph{flipping}~$e$ to~$f$. 
\end{definition}

Let~$(P,G)$ be  a triangulated polygon, see Definition~\ref{defn:triangulated-polygon}. We denote by $\xx_{D(G)}(P)\subset\xS_G(P)$ the collection of squared lengths labeled by the diagonals in~$D(G)$.

\begin{lem}
\label{lem:laurent-trim}
Let $G$ be a triangulated cycle, $e\!\in\! D(G)$ a diagonal in it, and \hbox{$G'\!=\!\mu_e(G)$.} Suppose that $\{i,j\}\notin D(G)$ is a diagonal such that $G$ is trimmed with respect to~$\{i,j\}$, but $G'$ is~not. Let $G''=\tau(G',i,j)$ be the trimming of $G'$ with respect to~$\{i,j\}$. Assume that the measurement $x_{ij} \in \xS(P)$ can be written as a Laurent polynomial in $\xS_{G''}(P)$ whose denominator is a monomial in $\xx_{D(G'')}(P)$. Then $x_{ij}$ can be expressed as a Laurent polynomial in $\xS_G(P)$ whose denominator is a monomial in~$\xx_{D(G)}(P)$. The same holds true with $\{i,j\}$ and $x_{ij}$ replaced by $(i,j,k)$ and $S_{ijk}$, respectively.
\end{lem}

\begin{proof}
Write $x_{ij} = Q_{G''}/M_{G''}$ where $Q_{G''}$ is a polynomial in $\xS_{G''}(P)$ and $M_{G''}$ a monomial in $\xx_{D(G'')}(P)$. Note that $D(G'') = D(G) \setminus \{e\}$. Also observe that $\xS_{G''}(P)$ consists of some subset of~$\xS_{G}(P)$ together with $x_f$ and two signed areas of the form $S_{fgh}$, for two triangles which have $f$ as a side. By \eqref{eq:f=}--\eqref{eq:s=}, each of these three measurements can be written as a Laurent polynomial in $\xS_G(P)$ with denominator $x_e \in \xS_G(P)$. Hence $Q_{G''}$ can be written as a Laurent polynomial in $\xS_G(P)$ with denominator a power of~$x_e$, and so $x_{ij}$ can be written as a Laurent polynomial in $\xS(G)$ with denominator a monomial in $\{x_e\}\cup \xx_{D(G'')}(P) = \xx_{D(G)}(P)$. A similar argument establishes the companion  result for $(i,j,k)$ and~$S_{ijk}$.
\end{proof}

\begin{prop}
\label{prop:laurent-xij}
Let $G$ be a triangulated $n$-gon which is trimmed with respect to a diagonal $\{i,j\}$. Then $x_{ij}$ can be written as a Laurent polynomial in the measurements $\xS_G(P)$ whose denominator is a monomial in the squared lengths of diagonals of $G$.
\end{prop}

\begin{proof}
We induct on $n$. The base $n\! =\! 4$ follows from Bretschneider's formula~\eqref{eq:bretschneider}. Let $n \!> \!4$. Note that no diagonal of $G$ is incident to~$i$; hence $e\!=\!\{i\!-\!1,i\!+\!1\} \in D(G)$. (Here and below we work modulo~$n$.) The triangulation $G'=\mu_e(G)$ includes a diagonal incident to vertex~$i$, hence is not trimmed with respect to~$\{i,j\}$. By Lemma~\ref{lem:laurent-trim}, it suffices to show that $x_{ij}$ has a Laurent expression in terms of  $\xS_{G''}(P)$ with denominator a monomial in~$\xx_{D(G'')}(P)$. Since $G''$ has fewer vertices than~$G$, we can invoke the induction hypothesis.
\end{proof}

\begin{prop}
\label{prop:laurent-sijk}
Let $G$ be a triangulated $n$-gon, trimmed with respect to a triple $(i,j,k)$. 
Then $S_{ijk}$ can be expressed as a Laurent polynomial in $\xS_G(P)$ whose denominator is a monomial in $\xx_{D(G)}(P)$.
\end{prop}

\begin{proof}
As in the proof of Proposition~\ref{prop:laurent-xij}, we induct on~$n$. For $n = 4$, the claim is immediate from equations \eqref{eq:r=}--\eqref{eq:s=}. 

For $n = 5$, all triangulations of a pentagon are equivalent up to cyclic renumbering of the vertices, so we can assume that $G$ has diagonals $\{1,3\}$ and $\{1,4\}$, see Figure~\ref{fig:laurent-triangulations}. Since $G$ is trimmed with respect to $(i,j,k)$, this triple must contain both~2~and~5. Applying Lemma~\ref{lem:laurent-trim} with $e\!=\!\{1,3\}$ and $(i,j,k)\!=\!(1,2,5)$ (resp., $(2,4,5)$), we conclude that $S_{125}$ (resp., $S_{245}$) can be written as a Laurent polynomial in $\xS_G(P)$, with denominator a monomial in $\xx_{D(G)}(P)\!=\!\{x_{13},x_{14}\}$. The case of $S_{235}$ is similar. 

Let us now consider the case when $n = 6$ and $G$ is the tri\-angulation with diagonals $\{1,3\}$, $\{3,5\}$, and $\{1,5\}$, see Figure~\ref{fig:laurent-triangulations}. The triples $(2,3,4)$, $(4,5,6)$, $(1,2,6)$ are contained in the triangulated pentagons $\{1,2,3,4,5\}$, $\{1,3,4,5,6\}$, and $\{1,2,3,5,6\}$ respectively; therefore $S_{234}$, $S_{456}$, and $S_{126}$ can be expressed as Laurent polynomials in $\xS_G(P)$, with denominator a monomial in $\xx_{D(G)}(P)$. 
The identity 
\[
S_{234}+S_{456}+S_{126} + S_{246} = S_{123}+S_{345}+S_{156}+S_{135}
\]
implies that $S_{246}$, too, can be expressed in such a form.

In general, suppose that $G$ includes a diagonal~$e$ incident to~$i$. Then $e$ crosses~$\{j,k\}$ (because $G$ is trimmed with respect to $(i,j,k)$). No diagonal of $G$ is incident to~$j$,~or else it would have to intersect $\{i,k\}$, hence $e$ as well. Thus $f\!=\!\{j\!-\!1,j\!+\!1\} \in D(G)$. Note that the triangulation $G'=\mu_f(G)$ is not trimmed with respect to $(i,j,k)$. Therefore by Lemma~\ref{lem:laurent-trim} and the induction hypothesis, $S_{ijk}$ can be expressed as a Laurent polynomial in $\xS_G(P)$ with denominator a monomial in $\xx_{D(G)}(P)$.

\begin{figure}[ht]
\begin{center}
\includegraphics[height=1.7in]{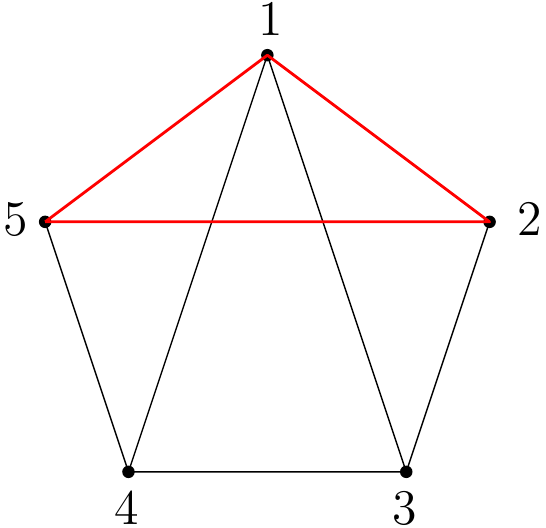} \hspace{0.3in}
\includegraphics[height=1.5in, trim=0 -4mm 0 5mm]{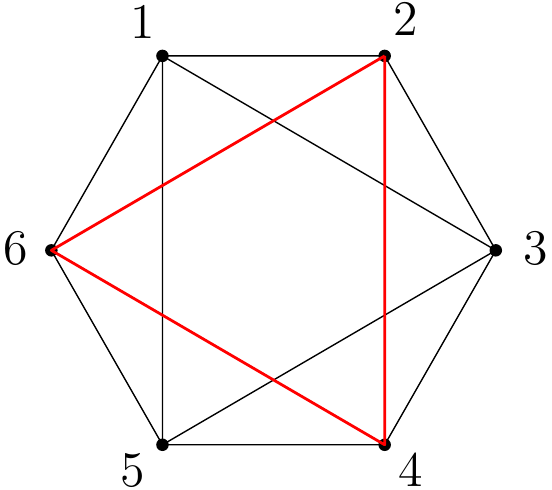}
\end{center}
\caption{Base cases appearing in the proof of Proposition~\ref{prop:laurent-sijk}.}
\vspace{-.2in}
\label{fig:laurent-triangulations}
\end{figure}

It remains to treat the case when no diagonal of $G$ incident to $i$, $j$, or $k$ exists. Then the diagonals $\{i-1,i+1\}$, $\{j-1,j+1\}$, $\{k-1,k+1\}$ must all appear in~$G$ (as before, we work modulo~$n$). Let $G'$ denote the triangulation obtained from $G$ by flipping $\{i-1,i+1\}$ to $\{i,v\}$, $v \in \{1,\dots,n\}$. If $\{i,v\}$ does not cross $\{j,k\}$, then Lemma~\ref{lem:laurent-trim} and the induction hypothesis apply. If $\{i,v\} \in E(G')$ crosses $\{j,k\}$, then so do $\{i-1,v\}, \{i+1,v\} \in E(G)$. It is then straightforward to check that unless 
\begin{equation}
\label{eq:exception-hexagon}
\{\{i-1,v\}, \{i+1,v\}\} = \{\{j-1,j+1\}, \{k-1,k+1\}\}, 
\end{equation}
flipping $\{j\!-\!1,j\!+\!1\}$ or $\{k\!-\!1,k\!+\!1\}$ will transform~$G$ into a triangulation that is not trimmed, in which case we can apply Lemma~\ref{lem:laurent-trim} and the induction hypothesis. 
In~the remaining case, condition \eqref{eq:exception-hexagon} forces $n\!=\!6$, with $G$ the triangulation shown~in Figure~\ref{fig:laurent-triangulations} and $\{i,j,k\}\!=\!\{2,4,6\}$, up to renumbering; this case was treated~above. 
\end{proof}

\begin{proof}[Proof of Theorem~\ref{thm:laurentness-full}]
In view of Remark~\ref{rem:can-trim}, the theorem this follows from Propositions~\ref{prop:laurent-xij} and~\ref{prop:laurent-sijk}. 
\end{proof}


The following algebraic statement strengthens Corollary~\ref{cor:frieze-is-recurrence}. 

\begin{cor}
\label{cor:Laurentness-Heronian-frieze}
Let $\pi$ be a traversing path. Let $\mathbf{z}_-$ denote a collection of complex numbers assigned to the indices in~$\pi$ which satisfy the appropriate Heron equations, and moreover the values at the integer nodes of~$\pi$ are nonzero. Then $\mathbf{z}_-$ can be extended to a Heronian frieze $\mathbf{z}$. Furthermore, each entry of $\mathbf{z}$ can be written as a Laurent polynomial in terms of~$\mathbf{z}_-$, with denominator a monomial in the values indexed by the integer nodes lying on~$\pi$.
\end{cor}

\pagebreak[3]

\begin{proof}
Let $G$ be the triangulated cycle corresponding to the path~$\pi$, 
cf.\ Remark~\ref{rem:path=triangulation}. 
In light of Corollary~\ref{cor:triangulation-recovery}, there exists a unique polygon~$P$ whose measurements in $\xS_G(P)$ match those in~$\mathbf{z}_-$. 
Now set $\mathbf{z}=\mathbf{z}(P)$ and apply Theorem~\ref{thm:laurentness-full}.
\end{proof}


Combining Corollary~\ref{cor:Laurentness-Heronian-frieze} with Corollary~\ref{cor:frieze-is-recurrence}, we obtain: 

\begin{cor}
\label{cor:Laurentness-Heronian-frieze-2}
Let $\mathbf{z}$ be a Heronian frieze satisfying the nonvanishing condition~\eqref{eq:nonzero-interior}. Then each entry of $\mathbf{z}$ can be written as a Laurent polynomial in terms of the entries lying on an arbitrary traversing path~$\pi$. The denominator of this Laurent polynomial is a monomial in the values indexed by the integer nodes lying on~$\pi$.
\end{cor}

\begin{example}
Let $\pi$ be the traversing path at the left rim of Figure~\ref{fig:heronian-frieze-example}. The values of the frieze lying on~$\pi$ are: 
\[
1, 2, 2, -2, 5, 2, 2, 2, 1; 1, 1, 1, 1
\]
(the last 4 values, all equal to~1, are associated with the dashed lines). 
In agreement with Corollary~\ref{cor:Laurentness-Heronian-frieze-2}, all values in the frieze are rational numbers whose denominators only have prime factors equal to 2 or~5. (In fact, the only denominator that shows up in this particular example is~5.) 
\end{example}

In the remainder of this section, we present an alternative approach to the Laurent phenomenon for Heronian friezes. While more technical than the proof of Theorem~\ref{thm:laurentness-full} given above, this approach produces a stronger (and more explicit) result. 

\begin{defn}
\label{defn:thin-triangulation}
A triangulation $G$ of an $n$-cycle with vertices $1,\dots,n$ is called \emph{thin} if it does not include three diagonals forming a triangle. By Remark~\ref{rem:trimmed-is-thin}, every trimmed triangulation is thin. Conversely, every thin triangulation~$G$ is trimmed with respect to a unique diagonal $\{b,c\}$, namely the diagonal connecting the only two vertices $b$ and~$c$ which are not incident to any diagonal in~$G$. 
\end{defn}

In the case of triangles, Remark~\ref{rem:can-trim} can be strengthened as follows. 

\begin{lem}
\label{lem:trim-triple}
Let $G$ be a triangulation trimmed with respect to a triple $(i,j,k)$. Suppose that for any $\ell\notin\{i,j,k\}$, the triangulation $G$~is trimmed with respect to at least one of the triples $(i,j,\ell)$, $(i,k,\ell)$, $(j,k,\ell)$. 
Then $G$ is trimmed with respect to at least one of $\{i,j\}$, $\{i,k\}$, or~$\{j,k\}$. In particular, $G$~is thin. 
%
\end{lem}

\begin{proof}
First suppose that one of the sides of the triangle $(i,j,k)$, say $\{i,j\}$, lies on the distinguished $n$-cycle; that is, $i\equiv j\pm 1 \bmod n$. (If there are two such sides, then $G$ is trimmed with respect to the third one.) It is easy to see that in this case, $G$ cannot simultaneously include (a) a diagonal that crosses $\{i,k\}$ but not $\{j,k\}$ and (b) a diagonal that crosses $\{j,k\}$ but not $\{i,k\}$. Hence~$G$ is trimmed with respect to either $\{i,k\}$ or~$\{j,k\}$. 

So let us assume that each of $\{i,j\}$, $\{i,k\}$, and~$\{j,k\}$ is a diagonal, i.e., is not an edge of the distinguished $n$-cycle. 
Let $D(G)[i,j]$ (resp., $D(G)[i,k]$, $D(G)[j,k]$) denote the subset of~$D(G)$ consisting of those diagonals in~$G$ which cross $\{i,j\}$ (resp., $\{i,k\}$, $\{j,k\}$). If one of these three subsets coincides with~$D(G)$, then we are done, so we can assume that none does. 

Since $G$ is trimmed with respect to $(i,j,k)$, we have $D(G)[i,j]\cup D(G)[i,k]\cup D(G)[j,k]=D(G)$. Suppose for a moment that $D(G)[i,j]\cap D(G)[i,k]=\varnothing$, i.e., no diagonal of~$G$ crosses both $\{i,j\}$ and~$\{i,k\}$. By assumption, there is a diagonal $D\in D(G)$ which does not cross~$\{j,k\}$. Say $D$ crosses $\{i,j\}$. Since $D$ crosses neither $\{i,k\}$ nor $\{j,k\}$, it must terminate at~$k$. Now, since none of the diagonals in $D(G)$ can cross~$D$, but each must cross one of the sides of $(i,j,k)$, it follows that each diagonal in~$G$ crosses~$\{i,j\}$, as desired. 

It remains to treat the case when the intersection $D(G)[i,j]\cap D(G)[i,k]$ is non\-empty, and moreover both $D(G)[i,j]\cap D(G)[j,k]$ and $D(G)[i,k]\cap D(G)[j,k]$ are nonempty as well. Let $D_i\in D(G)[i,j]\cap D(G)[i,k]$, $D_j\in D(G)[i,j]\cap D(G)[j,k]$, and $D_k\in D(G)[i,k]\cap D(G)[j,k]$. Then there exists a vertex~$\ell$ such that the diagonal $\{j,\ell\}$ crosses both $D_j$ and $\{i,k\}$ but neither $D_i$ nor~$D_k$. We now observe that the diagonal $D_i$ does not cross any of the sides of $(j,k,\ell)$, the diagonal $D_j$ does not cross any of the sides of $(i,k,\ell)$, and the diagonal $D_k$ does not cross any of the sides of $(i,j,\ell)$. In other words, the triangulation~$G$ is not trimmed with respect to each of the triples $(i,j,\ell)$, $(i,k,\ell)$, $(j,k,\ell)$, a contradiction. 
\end{proof}

\begin{rem}
\label{rem:reduce-to-trim}
We already noted, cf.\ Remark~\ref{rem:can-trim} and the proof of Theorem~\ref{thm:laurentness-full}, that it is sufficient to establish the Laurent phenomenon in the case when the triangulation~$G$ at hand is trimmed with respect to the measurement in question. In the case when the measurement is a squared distance~$x_{ij}$, this immediately implies that $G$ is thin. In the case of a signed area~$S_{ijk}$, we can assume that the triangulation~$G$, in addition to being trimmed with respect to~$(i,j,k)$, is also trimmed with respect to at least one of the triples $(i,j,\ell)$, $(i,k,\ell)$, $(j,k,\ell)$. (Otherwise, we can invoke the additive identity~\eqref{eq:split-ijkl} and then use trimming to induct on~$n$, the number of vertices.) 
Hence Lemma~\ref{lem:trim-triple} applies, meaning that we may assume that $G$ is trimmed with respect to one of the sides of $(i,j,k)$, and in particular is thin. 
\end{rem}

\begin{defn}
\label{defn:G-P-thin}
Let $G$ be a thin triangulation of an $n$-cycle, trimmed with respect to the diagonal~$D=\{c,n\}$, cf.\ Definitions~\ref{defn:trimmed} and~\ref{defn:thin-triangulation}. The $n-3$ diagonals of~$G$, together with $\{c-1,c\}$ and $\{1,n\}$, form the edge set of a spanning tree of~$G$. We~denote these $n-1$ edges by $D_2,\dots,D_n$, so that $D_2=\{c-1,c\}$, $D_n=\{1,n\}$, and for every $j\in\{2,\dots,n\}$, the edges $D_j$ and $D_{j+1}$ are two sides of a triangle in~$G$. 

Let $P=(A_1,\dots,A_n)$ be a polygon on the plane~$\AAA$. We continue to use the notation from Definition~\ref{defn:polygon} for the measurements $x_{ij}=x_{ij}(P)$ and $S_{ijk}=S_{ijk}(P)$. Let $v_2,\dots,v_n$ be the vectors corresponding to the edges $D_2,\dots,D_n$ of a thin triangulation~$G$ as described above; to be more precise, 
\begin{equation}
\label{eq:vk}
\text{$v_k=\overrightarrow{A_iA_j}$, where $D_k=\{i,j\}, i<j$.}
\end{equation}
We then define, for $2\le j \le n-1$: 
\begin{align}
\label{eq:Sj-via-vj}
S_j&=S_j(P)=2[v_j,v_{j+1}] , \\
T_j&=T_j(P)=2\langle v_j,v_{j+1}\rangle , 
\end{align}
It will also be helpful to introduce the following notation, for  $2\le a < b\le n$:
\begin{align}
\label{eq:sigma-even-ab}
\Sigma_{\textup{even}}(a,b) &= \sum_{\substack{J\subseteq\{a,\dots,b-1\}\\ |J| \text{ even}}} (-1)^{|J|/2}
\Bigl(\,\prod_{j \in J} S_j\Bigr)\Bigl(\,\prod_{\substack{a\le j< b \\j\notin J}}T_j\Bigr) ,
\\
\label{eq:sigma-odd-ab}
\Sigma_{\textup{odd}}(a,b) &= \sum_{\substack{J\subseteq\{a,\dots,b-1\}\\ |J| \text{ odd}}} (-1)^{(|J| - 1)/2}
\Bigl(\,\prod_{j \in J} S_j\Bigr)\Bigl(\,\prod_{\substack{a\le j< b \\j\notin J}}T_j\Bigr) .
\end{align}
\end{defn}

\begin{lem}
\label{lem:STsigma}
Every $S_j$, $T_j$, $\Sigma_{\textup{even}}(a,b)$ and~$\Sigma_{\textup{odd}}(a,b)$ is a polynomial with integer coefficients in the measurements in~$\xS_G(P)$. 
\end{lem}

\begin{proof}
First, $S_j\in\xS_G(P)$ since $S_j$ is the rescaled area of the triangle whose two sides are $D_j$ and~$D_{j+1}$ (cf.\ \eqref{eq:S(A,B,C)} and~\eqref{eq:Sijk(P)}). 
Second, note that $v_j-v_{j+1}$ is a vector linking two adjacent points on the perimeter of the polygon~$P$. Consequently, 
\[
T_j=\langle v_j-v_{j+1},v_j-v_{j+1}\rangle-\langle v_j,v_j\rangle-\langle v_{j+1},v_{j+1}\rangle  \in \ZZ[\xS_G(P)].
\]
The statement concerning $\Sigma_{\textup{even}}(a,b)$ and~$\Sigma_{\textup{odd}}(a,b)$ follows.
\end{proof}

\begin{prop}
\label{prop:brackets-thin}
In the notation of \eqref{eq:vk} and  \eqref{eq:sigma-even-ab}--\eqref{eq:sigma-odd-ab}, we have: 
\begin{align}
\label{eq:angle-bracket-formulas}
\langle v_a, v_b\rangle &= 2^{a-b}\,\Sigma_{\textup{even}}(a,b)\prod_{m=a+1}^{b-1}\langle v_m,v_m\rangle^{-1} 
,\\
\label{eq:square-bracket-formulas}
[v_a,v_b]&=2^{a-b}\,\Sigma_{\textup{odd}}(a,b)\prod_{m=a+1}^{b-1}\langle v_m,v_m\rangle^{-1} 
. 
\end{align}
In particular, both $\langle v_a, v_b\rangle$ and $[v_a,v_b]$ are Laurent polynomials with integer coefficients in the measurements  in~$\xS_G(P)$. 
In each of these Laurent polynomials, the denominator is a square-free product of the measurements $x_{ij}\in\xx_{D(G)}(P)$. 
\end{prop}

\begin{proof}
Let us adjoin a formal square root $\varepsilon=\sqrt{-1}$ to~$\CC$. In other words, our computations will be done in the ring $\CC[\varepsilon]/\langle\varepsilon^2+1\rangle$. The key observation is that for~$a < b$,
\begin{equation}
\label{eq:prod(Tk+eps-Sk):1-thin}
\prod_{m=a}^{b-1} (T_m + \varepsilon S_m) = \Sigma_{\textup{even}}(a,b) + \varepsilon \Sigma_{\textup{odd}}(a,b).
\end{equation}
Furthermore, with the notation $v_m = \left[\begin{smallmatrix}v_m'\\ v_m''\end{smallmatrix}\right]$, we have 
\[
T_m+\varepsilon S_m = 2(\langle v_m,v_{m+1}\rangle + \varepsilon [v_m, v_{m+1}]) = 2(v_m'-\varepsilon v_m'')(v_{m+1}'+\varepsilon v_{m+1}'')
\]
and consequently
\begin{align}
\notag
\prod_{k=a}^{b-1}(T_m+\varepsilon S_m) &= 2^{b-a}\prod_{m=a}^{b-1}(v_m'-\varepsilon v_m'')(v_{m+1}'+\varepsilon v_{m+1}'')\\[-1pt]
\notag
&=2^{b-a}(v_a'-\varepsilon v_a'')(v_b'+\varepsilon v_b'')\prod_{m=a+1}^{b-1}(v_m'-\varepsilon v_m'')(v_m'+\varepsilon v_m'')\\[-1pt]
\label{eq:prod(Tk+eps-Sk):2-thin}
&= 2^{b-a}(\langle v_a,v_b\rangle +\varepsilon[v_a,v_b]) \prod_{m=a+1}^{b-1}\langle v_m, v_m\rangle.
\end{align}
Comparing~\eqref{eq:prod(Tk+eps-Sk):1-thin} with~\eqref{eq:prod(Tk+eps-Sk):2-thin}, we conclude that
\begin{align}
\label{eq:angle-brackets-thin}
\Sigma_{\textup{even}}(a,b)&=
2^{b-a}\langle v_a, v_b\rangle \prod_{m=a+1}^{b-1}\langle v_m,v_m\rangle  ,
\\[-1pt]
\label{eq:square-brackets-thin}
\Sigma_{\textup{odd}}(a,b)&=
2^{b-a}[v_a,v_b] \prod_{m=a+1}^{b-1}\langle v_m,v_m\rangle  .
\end{align}
Rearranging equations~\eqref{eq:angle-brackets-thin} and~\eqref{eq:square-brackets-thin} yields \eqref{eq:angle-bracket-formulas} and~\eqref{eq:square-bracket-formulas}.
\end{proof}

For $1\le j<k\le n$, consider the unique path in the spanning tree formed by $D_2,\dots,D_n$ which connects $j$ to~$k$. We denote the length of this path by~$\ell(j,k)$. Let $i_1(j,k)\le \cdots\le i_{\ell(j,k)}(j,k)$ be the indices of the edges forming this path, so that the set of these edges is $\{D_{i_a} : 1\le a\le\ell(j,k)\}$. 


\begin{prop}
\label{prop:xjk-Sijk-Laurent-thin}
Let $(P,G)$ be a thin triangulation of a plane $n$-gon, trimmed with respect to the diagonal~$D=\{c,n\}$, see Definition~\ref{defn:G-P-thin}. 
Then 
\begin{align}
\label{eq:xjk-laurent-formulas}
x_{jk} &= \sum_{a=1}^{\ell(j,k)} \sum_{b=1}^{\ell(j,k)} (-1)^{a+b}\langle v_{i_a(j,k)},v_{i_b(j,k)}\rangle \quad (1 \le j < k < n), \\
\label{eq:Sjkn-laurent-formulas}
S_{jkn} &= 2\,\varkappa_j \varkappa_k \sum_{a = 1}^{\ell(j,k)} \sum_{b=1}^{\ell(k,n)}(-1)^{a+b}[v_{i_a(j,k)},v_{i_b(k,n)}] \quad (1 \le j < k < n),
\end{align}
where 
we use the notation
\[
\varkappa_j=\biggl\{\!\! \begin{array}{ll} 
+1 &\text{if } c\le j;\\[.05in]
-1 &\text{if } c>j.
\end{array}
\]
\end{prop}

\begin{proof}
Observe that
$
\overrightarrow{A_jA_k} = \varkappa_j \sum_{a=1}^{\ell(j,k)}(-1)^av_{i_a(j,k)}$.
Therefore
\begin{align*}
x_{jk} &= \langle \overrightarrow{A_jA_k},\overrightarrow{A_jA_k}\rangle = \biggl\langle\, \sum_{a = 1}^{\ell(j,k)} (-1)^av_{i_a(j,k)},\sum_{b = 1}^{\ell(j,k)} (-1)^b v_{i_b(j,k)}\biggr\rangle,\\
S_{jkn} &= 2\,[\overrightarrow{A_jA_k},\overrightarrow{A_kA_n}] 
= 2\,\varkappa_j \varkappa_k \biggl[\,\sum_{a=1}^{\ell(j,k)} (-1)^av_{i_a(j,k)}, \sum_{b = 1}^{\ell(k,n)} (-1)^b v_{i_b(k,n)}\biggr].
\end{align*}
Now the bilinearity of the forms $\langle\cdot,\cdot\rangle$ and $[\cdot,\cdot]$ implies \eqref{eq:xjk-laurent-formulas}--\eqref{eq:Sjkn-laurent-formulas}. 
\end{proof}

\begin{thm}
\label{thm:laurent-thin}
Let $(P,G)$ be a triangulated polygon in the plane~$\AAA$. Every measurement $x_{ij}$ (resp.,~$S_{ijk}$) in $\xS(P)$ can be expressed as a Laurent polynomial with integer coefficients in the measurements in~$\xS_G(P)$, with the denominator equal to the product of the measurements $x_{ab}$ corresponding to the diagonals $\{a,b\}\in D(G)$ which cross $\{i,j\}$ (resp.~$(i,j,k)$). 
\end{thm}

\begin{proof}
In the case when triangulation~$G$ is thin, and trimmed with respect to the measurement in question, the statement follows by substituting the formulas in Proposition~\ref{prop:brackets-thin} into the ones given in Proposition~\ref{prop:xjk-Sijk-Laurent-thin}, and recalling Lemma~\ref{lem:STsigma}. 
The general case follows by Remark~\ref{rem:reduce-to-trim} (based on Lemma~\ref{lem:trim-triple}). 
As noted in the latter remark, 
the final formulas for the signed areas $S_{ijk}$ are obtained by using the additive relations~\eqref{eq:split-ijkl} to 
combine several expressions resulting from \eqref{eq:Sjkn-laurent-formulas} and \eqref{eq:square-bracket-formulas}. 
\end{proof}

\begin{rem}
For the readers interested in the computational aspects of these problems, we note that the above formulas lead to polynomial-complexity algorithms: although the sums in~\eqref{eq:angle-bracket-formulas}--\eqref{eq:square-bracket-formulas} may contain exponentially many terms as $n\to\infty$, they can be computed very fast via the product formula~\eqref{eq:prod(Tk+eps-Sk):1-thin}.
\end{rem}

We conclude this section by providing an explicit version of Theorem~\ref{thm:laurentness-full} for the ``fan'' triangulation in which all diagonals are incident to a single vertex. 

\begin{defn}
\label{defn:G1}
Let $G_1$ be the triangulated $n$-cycle with vertices $1,2,\dots,n$, sides $\{1,2\},\dots,\{n-1,n\},\{n,1\}$, and diagonals $\{1,3\}, \dots, \{1,n-1\}$.
Let $P\!=\!(A_1,\dots,A_n)$ be a polygon on the plane~$\AAA$. 
We continue to use the notation~\eqref{eq:xij(P)}--\eqref{eq:Sijk(P)}. 
For $1<a<b\le n$ and $J\subset\{a,\dots,b-1\}$, we denote
\begin{alignat}{3}
\label{eq:QJab}
Q_{J,[a,b]} &= \Bigl(\,\prod_{j \in J} S_{1,j,j+1}\Bigr)\Bigl(\,\prod_{j \in \{a,\dots,b-1\} \setminus J} (x_{1j}+x_{1,j+1}-x_{j,j+1}) \Bigr). 
\end{alignat}
\end{defn}

\begin{cor}
\label{cor:laurent-fan}
Each measurement in $\xS(P)$ can be explicitly expressed as a Laurent polynomial in the measurements in~$\xS_{G_1}(P)$ (cf.\ Definition~\ref{defn:G1}), as follows. For $1<a<b\le n$, we have: 
\begin{align}
\label{eq:xij-no-angle-brackets}
x_{ab}&=x_{1a}+x_{1b} -  \B(2^{b-a-1}\prod_{k=a+1}^{b-1}x_{1k}\B)^{\!\!-1} \sum_{\substack{J\subset\{a,\dots,b-1\}\\ \textup{$|J|$ even}}} (-1)^{|J|/2} \,Q_{J,[a,b]} 
,\\
\label{eq:S1ij}
S_{1ab}&= \B(2^{b-a-1}\prod_{k=a+1}^{b-1}x_{1k}\B)^{-1} 
\sum_{\substack{J\subset\{a,\dots,b-1\}\\ \textup{$|J|$ odd}}} (-1)^{(|J|-1)/2} \,Q_{J,[a,b]} 
.\\
\intertext{Also, for $1<a<b<c\le n$, we have:}
\label{eq:Sijk=S1ij+...}
S_{abc}&=S_{1ab}+S_{1bc}-S_{1ac} . 
\end{align}
\end{cor}

\begin{proof}
Formula \eqref{eq:Sijk=S1ij+...} is clear. Formula \eqref{eq:xij-no-angle-brackets} is a special case of \eqref{eq:xjk-laurent-formulas} (with the substitutions~\eqref{eq:angle-bracket-formulas}), applied to the trimming of~$G$ with respect to the diagonal~$\{a,b\}$. Similarly, formula \eqref{eq:S1ij} is a special case of \eqref{eq:Sjkn-laurent-formulas}, with the substitutions~\eqref{eq:square-bracket-formulas}. 
\end{proof}

\pagebreak[3]

\section{\CM friezes}
\label{sec:cayley-menger}

Recall the following classical result; 
see, e.g., \cite[Section~9.7.3]{berger-geometry-v1} or \cite[Section~40]{blumenthal}.

\begin{prop}
\label{prop:cayley-menger-plane}
Let $(A_1,A_2,A_3,A_4)$ be a quadrilateral in~$\AAA$ with measurements 
\begin{equation}
\label{eq:abcdef-1234}
a=x_{14}\,,\quad 
b=x_{12}\,,\quad 
c=x_{23}\,,\quad 
d=x_{34}\,, \quad
e=x_{13}\,,\quad
f=x_{24}
\end{equation}
(squared distances between all pairs of vertices), as in Example~\ref{example:quadrilateral}. 
Then 
\begin{equation}
\label{eq:cayley-menger}
M(a,b,c,d,e,f)\stackrel{\rm def}{=}
\det\begin{bmatrix} 0 & 1 & 1 & 1 & 1 \\ 1 & 0 & b & e & a \\ 1 & b & 0 & c & f \\ 1 & e & c & 0 & d \\ 1 & a & f & d & 0 \end{bmatrix}=0. 
\end{equation}
\end{prop}

The determinant $M(a,b,c,d,e,f)$ appearing in~\eqref{eq:cayley-menger} is called the \emph{Cayley-Menger determinant}. It is a homogenous polynomial of degree~$3$ in the variables $a,b,c,d,e,f$. Viewed as a polynomial in each individual variable, it has degree~$2$. For each $5$-tuple of numbers $a,b,c,d,e$, there are ordinarily two values of~$f$ such that \eqref{eq:cayley-menger} holds. Informally speaking, a configuration of four points in the plane~$\AAA$ is ``almost'' determined by five of the six squared lengths between these points, up to a binary~choice. 


\begin{defn}
\label{defn:CM-diamond}
A \emph{\CM diamond} is a 6-tuple $(a,b,c,d,e,f)$ of complex numbers satisfying~\eqref{eq:cayley-menger}. As with Heronian diamonds, we typically arrange these six numbers in a diamond pattern, as shown in Figure~\ref{fig:CM-diamond}.
Proposition~\ref{prop:cayley-menger-plane} can be restated as saying that for any plane quadrilateral $(A_1,A_2,A_3,A_4)$, the six squared distances listed in~\eqref{eq:abcdef-1234} form a \CM diamond. Cf.\ Remark~\ref{rem:heronian-diamond}. 
\end{defn}

\begin{figure}[ht]
\begin{center}
\vspace{-.1in}
\includegraphics[scale=0.7]{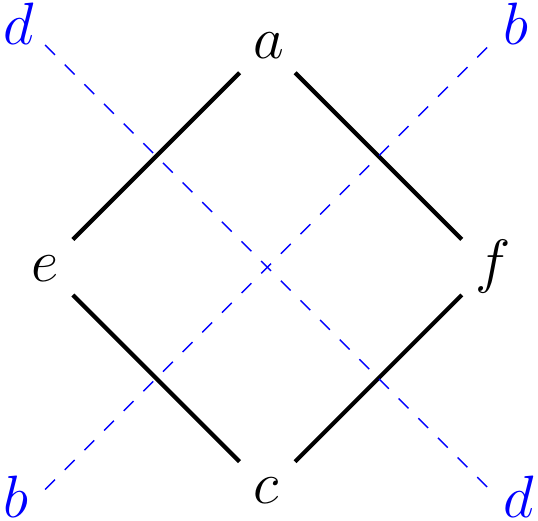}
\end{center}
\vspace{-.1in}
\caption{A \CM diamond, cf.\ Definition~\ref{defn:CM-diamond}.}
\label{fig:CM-diamond}
\end{figure}

\vspace{-.2in}

\begin{rem}
\label{rem:CM-dihedral}
Thanks to the symmetries of the \CM determinant, the notion of a \CM diamond is invariant under the dihedral symmetries of the underlying square pattern. Thus, if $(a,b,c,d,e,f)$ is a \CM diamond, then so are $(c,d,a,b,e,f)$, $(a,d,c,b,f,e)$, $(f,d,e,b,a,c)$, etc. 
\end{rem}



\begin{defn}
Let $n \geq 4$ be an integer. A \emph{\CM frieze} $\mathbf{z} = (z_\alpha)_{\alpha \in \InCM}$ is an array of complex numbers indexed by the set 
\[
\InCM = \{(i,j)\in\ZZ^2 : 0\le j-i\le n\} \cup L_n
\]
(see \eqref{eq:Ln}) in which, for every $(i,j) \in \ZZ^2$ satisfying~$1 \leq j-i \leq n-1$, the 6-tuple 
\begin{equation}
\label{eq:xdiamond}
\xdiamond_{i,j}(\mathbf{z})
\stackrel{\rm def}{=}(z_{(i,j+1)}, z_{(i+\frac12,\smallneline)}, z_{(i+1,j)}, z_{(\smallseline,j+\frac12)}, z_{(i,j)}, z_{(i+1,j+1)})
\end{equation}
(see Figure~\ref{fig:CM-frieze-blowup}) forms a \CM diamond. 
In other words, we require that 
\[
M(\xdiamond_{i,j}(\mathbf{z}))=0 \quad (1 \leq j-i \leq n-1). 
\]
In addition, we impose the following boundary conditions (cf.~\eqref{eq:partial-diamond-1}--\eqref{eq:partial-diamond-2}): 
\begin{alignat}{5}
\label{eq:cm-boundary-1}
z_{(i,i)}=0, \quad &&z_{(i,i+1)} &= z_{(i+\frac12,\smallneline)} = z_{(\smallseline,i+\frac12)}\quad &&(i \in \ZZ),\\
\label{eq:cm-boundary-2}
z_{(i,i+n)}=0, \quad && z_{(i,i+n-1)} &= z_{(i-\frac12,\smallneline)} = z_{(\smallseline, i+n-\frac12)}\quad &&(i \in \ZZ). 
\end{alignat}
An example of a \CM frieze is shown in Figure~\ref{fig:coherent-frieze-example}. 
\end{defn}

\begin{figure}[ht]
\begin{center}
\vspace{-.1in}
\includegraphics[scale=0.6]{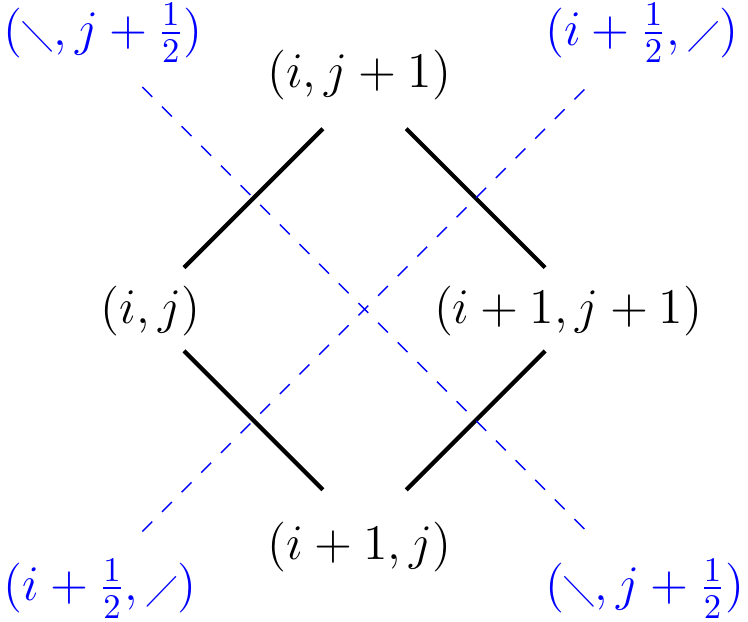}
\vspace{-.2in}
\end{center}
\caption{Indexing set for a diamond in a \CM frieze. 
}
\label{fig:CM-frieze-blowup}
\end{figure}
\vspace{-.2in}

\begin{figure}[ht]
\begin{center}
\vspace{-.1in}
\includegraphics[scale=0.65]{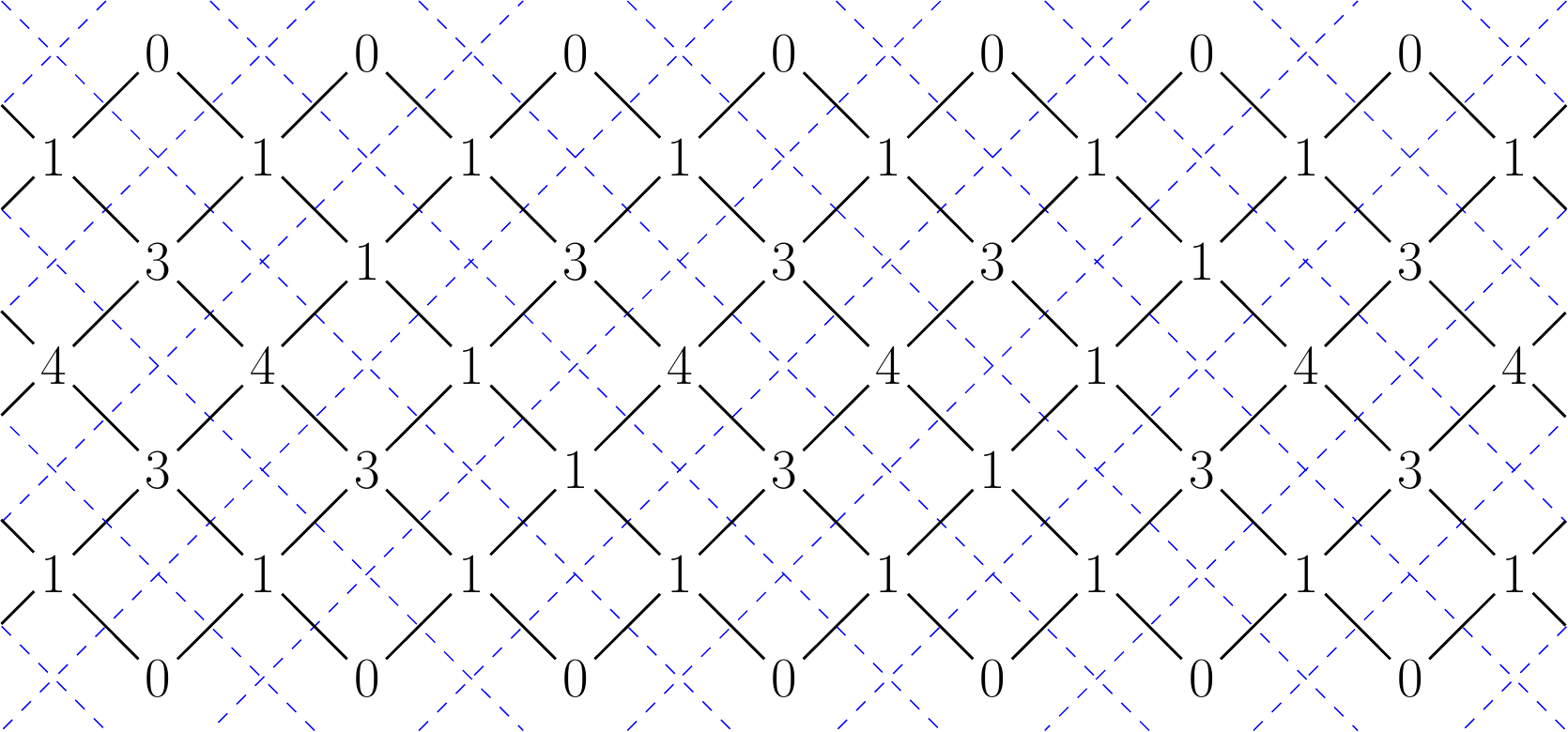}
\end{center}
\caption{A \CM frieze $\mathbf{z}$ of order~6. The entries associated with the dashed lines (which match the entries in the top and bottom nonzero rows) are all equal to~1. The leftmost entries in each row, starting from the bottom, are indexed by $(4,4)$, $(3,4)$, $(3,5)$, $(2,5)$, $(2,6)$, $(1,6)$, and $(1,7)$, respectively.}
\label{fig:coherent-frieze-example}
\end{figure}

\begin{defn}
\label{defn:CM-Frieze-from-P}
By Proposition~\ref{prop:cayley-menger-plane}, any $n$-gon $P$ in the plane $\AAA$ gives rise to a \CM frieze ${\mathbf{z}} = {\mathbf{z}}_{\textup{CM}}(P)$ of order $n$ by setting
\begin{align}
z_{(i,j)} &= x_{\langle i\rangle\langle j\rangle},\\
z_{(i+\frac12,\smallneline)} &= x_{\langle i\rangle\langle i+1\rangle},\\
z_{(\smallseline, j+\frac12)} &= x_{\langle j\rangle\langle j+1\rangle},
\end{align}
where $\langle m\rangle$ denotes the unique integer in $\{1,\dots,n\}$ satisfying $m\equiv\langle m\rangle\pmod n$. The boundary conditions \eqref{eq:cm-boundary-1}--\eqref{eq:cm-boundary-2} are easily checked, using the fact that $x_{ii} = 0$ for all $i \in \{1,\dots,n\}$. 

An example is shown in Figure~\ref{fig:nonconvex-hexagon}. 
\end{defn}

\begin{figure}[ht]
\begin{center}
\vspace{-.05in}
\setlength{\unitlength}{1.5pt}
\begin{picture}(56,52)(0,0)
\thicklines
\multiput(15,0)(0,52){2}{\line(1,0){26}}
\multiput(0,26)(26,0){2}{\line(15,26){15}}
\multiput(0,26)(26,0){2}{\line(15,-26){15}}

\multiput(15,0)(0,52){2}{\circle*{1.5}}
\multiput(41,0)(0,52){2}{\circle*{1.5}}
\multiput(0,26)(26,0){2}{\circle*{1.5}}

\put(33,26){\makebox(0,0){$A_1$}}
\put(46,0){\makebox(0,0){$A_2$}}
\put(46,52){\makebox(0,0){$A_6$}}
\put(9,0){\makebox(0,0){$A_3$}}
\put(8.5,52){\makebox(0,0){$A_5$}}
\put(-6,26){\makebox(0,0){$A_4$}}

\end{picture}
\end{center}
\caption{A hexagon $P=(A_1,\dots,A_6)$ giving rise to the frieze in Figure~\ref{fig:coherent-frieze-example}. 
Each side of the hexagon is of length~1, as are the diagonals $A_1A_3$, $A_1A_4$,  $A_1A_5$. 
The measurements $x_{44}=0$, $x_{34}=1$, $x_{35}=3$, $x_{25}=4$, $x_{26}=3$, $x_{16}=1$, $x_{11}=0$ match the leftmost entries in each row of the frieze. 
}
\label{fig:nonconvex-hexagon}
\end{figure}
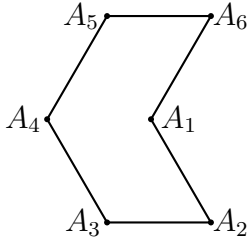
\vspace{-.2in}

\begin{definition}
It will be helpful to introduce some nonconventional (but suggestive) notation for the partial derivatives of the \CM polynomial $M=M(a,b,c,d,e,f)$ with respect to its 6~arguments. This notation makes reference to the placement of these arguments in the diamond, cf.\ Figure~\ref{fig:CM-diamond}. We denote 
\begin{alignat}{7}
\notag
&&\partial_\uparrow M &=\frac{\partial M}{\partial a} \\[-2pt]
\label{eq:partial-arrows}
\partial_\leftarrow M =\frac{\partial M}{\partial e} \quad&&& && \quad \partial_\rightarrow M =\frac{\partial M}{\partial f} \\[-2pt]
\notag
&& \partial_\downarrow M &=\frac{\partial M}{\partial c} \\[-2pt]
\label{eq:partial-slanted}
\partial_\smallneline M =\frac{\partial M}{\partial b}  \qquad\qquad\qquad &&&&& \qquad \qquad \qquad
\partial_\smallseline M =\frac{\partial M}{\partial d}.
\end{alignat}
In particular, 
\begin{align*}
\partial_\rightarrow M(a,b,c,d,e,f)
&=\frac{\partial M}{\partial f} (a,b,c,d,e,f) \\[-1pt]
&=2(-ab + ac + bd - cd + ae + be + ce + de - e^2 - 2ef) \\
&=2Q(a,b,c,d,e,f), 
\end{align*}
where
\begin{equation}
\label{eq:Q(abcdef)}
Q(a,b,c,d,e,f)=(a-d)(c-b)+e(a+b+c+d-e-2f). 
\end{equation}
The other five partial derivatives of~$M$ are obtained by evaluating $2Q$ at the appropriate permutations of $(a,b,c,d,e,f)$. 
(Incidentally, some permutations leave $Q$ invariant: $Q(a,b,c,d,e,f) = Q(c,d,a,b,e,f)=Q(b,a,d,c,e,f)=Q(d,c,b,a,e,f)$.)  
\end{definition}


When $a,b,c,d,e,f$ satisfy equation \eqref{eq:cayley-menger}, there are alternative formulas for the squared partial derivatives of~$M$, see Lemma~\ref{lem:noncoh-discriminant} below. Put in a different way, the formulas in Lemma~\ref{lem:noncoh-discriminant} hold modulo~$M$. 

\begin{lem}
\label{lem:noncoh-discriminant}
For a Cayley-Menger diamond $(a,b,c,d,e,f)$, we have: 
\begin{align}
\label{eq:noncoh-f}
(\partial_\rightarrow M(a,b,c,d,e,f))^2 &= 4H(b,c,e)H(a,d,e),\\
\label{eq:noncoh-a}
(\partial_\uparrow M(a,b,c,d,e,f))^2 &= 4H(b,c,e)H(c,d,f),\\
\label{eq:noncoh-c}
(\partial_\downarrow M(a,b,c,d,e,f))^2 &= 4H(a,b,f)H(a,d,e),\\
\label{eq:noncoh-e}
(\partial_\leftarrow M(a,b,c,d,e,f))^2 &= 4H(a,b,f)H(c,d,f).
\end{align}
\end{lem}
(There are also analogous formulas for $\partial_\smallneline M$ and~$\partial_\smallseline M$.)

\begin{proof}
It is straightforward to verify the polynomial identity
\begin{equation}
\label{eq:M-and-partial-M}
(\partial_\rightarrow M(a,b,c,d,e,f))^2 = -8e M(a,b,c,d,e,f) +4H(b,c,e)H(a,d,e). 
\end{equation}
For a \CM diamond, we have $M(a,b,c,d,e,f)=0$, and \eqref{eq:noncoh-f} follows. 
%
Formulas \eqref{eq:noncoh-a}--\eqref{eq:noncoh-e} are proved in a similar way; alternatively, use the symmetry of $M$ under the natural $S_4$-action. 
%
%
\end{proof}

When a \CM diamond consists of the measurements coming from a plane quadrilateral, Lemma~\ref{lem:noncoh-discriminant} can be strengthened by assigning a geometric meaning to each evaluation of a partial derivative of~$M$:  

\begin{prop}[{\rm \cite[p.~40]{dziobek1900}; cf.~\cite[Theorem~1]{kpsz2017}}]
\label{prop:partial-menger-formulas}
Let $(A_1,A_2,A_3,A_4)$ be a quadri\-lateral in~$\AAA$. We continue to use notation \eqref{eq:xij(P)}, \eqref{eq:Sijk(P)}, \eqref{eq:abcd-ABCD}, \eqref{eq:pqrs-ABCD}, \eqref{eq:partial-arrows}, \eqref{eq:partial-slanted}.
Let 
\[
\mathbf{x}=(x_{14}\,,x_{12}\,,x_{23}\,,x_{34}\,,x_{13}\,,x_{24}) 
=(a,b,c,d,e,f)
\]
denote the corresponding Cayley-Menger diamond. Then
\begin{alignat}{5}
\label{eq:partial-13,24}
\partial_\leftarrow M(\xx) &= -2S_{124}S_{234}, \qquad 
& \partial_\rightarrow M(\xx) &= -2S_{123}S_{134}, 
\\
\label{eq:partial-12,34}
\partial_\smallneline M(\xx) &= 2S_{134}S_{234}, \qquad 
& \partial_\smallseline M(\xx) &=2S_{123}S_{124}, 
\\
\label{eq:partial-14,23}
\partial_\uparrow M(\xx) &=2S_{123}S_{234} , \qquad 
& \partial_\downarrow M(\xx) &=2S_{124}S_{134} .
\end{alignat}
\end{prop}

\begin{prop}
\label{prop:noncoh-equation-123456}
Let 
\begin{align}
\label{eq:x1-hexagon}
\mathbf x_1 &= (x_{15},x_{12},x_{24},x_{45},x_{14},x_{25}),\\
\label{eq:x2-hexagon}
\mathbf x_2 &= (x_{16},x_{12},x_{25},x_{56},x_{15},x_{26}),\\
\label{eq:x3-hexagon}
\mathbf x_3 &= (x_{25},x_{23},x_{34},x_{45},x_{24},x_{35}),\\
\label{eq:x4-hexagon}
\mathbf x_4 &= (x_{26},x_{23},x_{35},x_{56},x_{25},x_{36})
\end{align}
be \CM diamonds (see Figure~\ref{fig:3-by-3-variables-hexagon} for a visual representation). Then
\begin{equation}
\label{eq:noncoherence-123456}
\bigl(\partial_\leftarrow M(\mathbf x_1)\,\partial_\rightarrow M(\mathbf x_4)\bigr)^2 = \bigl(\partial_\uparrow M(\mathbf x_2)\,\partial_\downarrow M(\mathbf x_3)\bigr)^2.
\end{equation}
Moreover, if $x_{ij} \!=\! x_{ij}(P)$ are the measurements of a hexagon $P \!= \!(A_1, \dots, A_6)$,~then
\begin{equation}
\label{eq:hexagon-coherence}
\partial_\leftarrow M(\mathbf x_1)\,\partial_\rightarrow M(\mathbf x_4) = \partial_\uparrow M(\mathbf x_2)\,\partial_\downarrow M(\mathbf x_3).
\end{equation}
\end{prop}


\begin{figure}[ht]
\begin{center}
\includegraphics[scale=0.65]{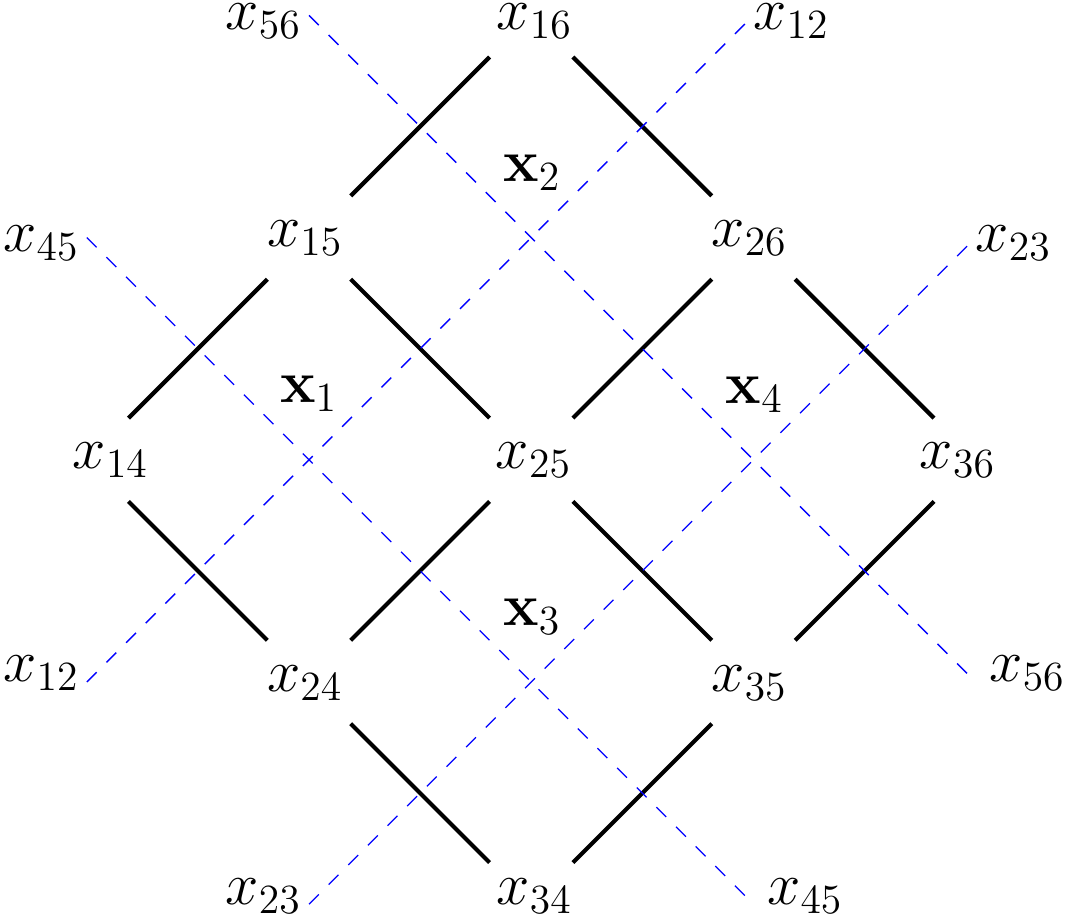}
\vspace{-.1in}
\end{center}
\caption{The variables involved in Propositions~\ref{prop:noncoh-equation-123456} and~\ref{prop:(non)coh-equations}.}
\vspace{-.2in}
\label{fig:3-by-3-variables-hexagon}
\end{figure}

\begin{proof}
Applying Lemma~\ref{lem:noncoh-discriminant} to each \CM diamond $\mathbf x_i$, we conclude that 
\[
\bigl(\partial_\leftarrow M(\mathbf x_1)\partial_\rightarrow M(\mathbf x_4)\bigr)^2 =16H_{245}H_{125}H_{256}H_{235} = \bigl(\partial_\uparrow M(\mathbf x_2)\partial_\downarrow M(\mathbf x_3)\bigr)^2,
\]
where
\begin{alignat*}{4}
H_{125} &= H(x_{12},x_{15},x_{25}),\quad & H_{256} &= H(x_{25},x_{26},x_{56}), \\
H_{245} &= H(x_{24},x_{25},x_{45}),\quad & H_{235} &= H(x_{23},x_{25},x_{35}). 
\end{alignat*}
If the $x_{ij}$ come from a hexagon $(A_1, \dots, A_6)$, then Proposition~\ref{prop:partial-menger-formulas} applies, and 
\[
\partial_\leftarrow M(\mathbf x_1)\partial_\rightarrow M(\mathbf x_4) = 4S_{245}S_{125}S_{256}S_{235} = \partial_\uparrow M(\mathbf x_2)\partial_\downarrow M(\mathbf x_3).\qedhere
\]
\end{proof}

\begin{rem}
The last assertion in Proposition~\ref{prop:noncoh-equation-123456} can be restated in a more explicit form, using notation~\eqref{eq:Q(abcdef)}. Let $P \!= \!(A_1, \dots, A_6)$ be a hexagon on the plane~$\AAA$. As before, let $x_{ij}$ denote the squared distance between vertices $A_i$ and~$A_j$.~Then 
\begin{align}
\label{eq:coherence-Q-123456}
&Q(x_{26},x_{23},x_{35},x_{56},x_{25},x_{36}) \,
Q(x_{24}, x_{12}, x_{15}, x_{45}, x_{25}, x_{14}) \\
\notag
=\,&Q(x_{15},x_{56},x_{26},x_{12},x_{25},x_{16})\,
Q(x_{35},x_{45},x_{24}, x_{23},x_{25},x_{34}). 
\end{align}
Equation~\eqref{eq:coherence-Q-123456} involves 13 squared distances~$x_{ij}$, with the exception of $x_{13}$ and~$x_{46}$. 
\end{rem}

Let $\mathbf{z}$ be a \CM frieze. Consider four adjacent diamonds sharing a common vertex~$(i,j)$, as shown in Figure~\ref{fig:3-by-3-variables}.
Proposition~\ref{prop:noncoh-equation-123456} implies that for any $(i,j) \in \ZZ^2$ with $2 \leq j-i \leq n-2$, we have
\begin{equation}
\label{eq:noncoherence}
(\partial_\leftarrow M(\xdiamond_{i-1,j-1}(\mathbf{z}))\,
\partial_\rightarrow M(\xdiamond_{i,j}(\mathbf{z})))^2
=(\partial_\uparrow M(\xdiamond_{i-1,j}(\mathbf{z}))\,
\partial_\downarrow M(\xdiamond_{i,j-1}(\mathbf{z})))^2, 
\end{equation}
where we use notation~\eqref{eq:xdiamond} for the diamonds of~$\mathbf{z}$. 
Consequently, 
\begin{equation}
\label{eq:+-coherence}
\partial_\leftarrow M(\xdiamond_{i-1,j-1}(\mathbf{z}))\,
\partial_\rightarrow M(\xdiamond_{i,j}(\mathbf{z}))
=\pm 
\partial_\uparrow M(\xdiamond_{i-1,j}(\mathbf{z}))\,
\partial_\downarrow M(\xdiamond_{i,j-1}(\mathbf{z})). 
\end{equation}
In general, the signs of the products appearing on both sides of ~\eqref{eq:+-coherence} do not have to match, see Example~\ref{example:noncoherent} below. The settings where they do match play a key role in this section. 

\begin{figure}[ht]
\includegraphics[scale=0.65]{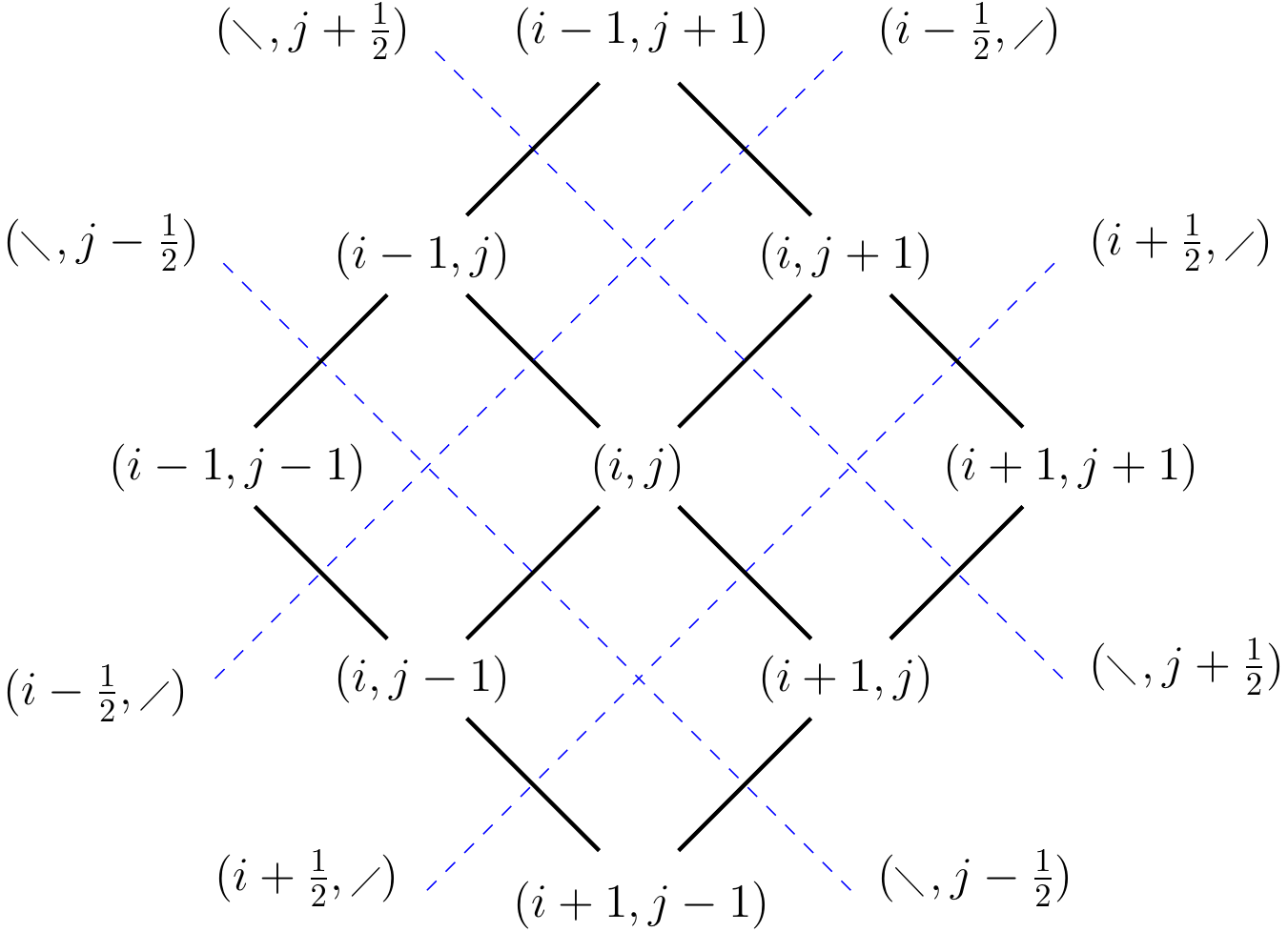}
\vspace{-.1in}
\caption{The fragment of a \CM frieze involved in Definition~\ref{defn:coherence-equation}. 
}
\vspace{-.2in}
\label{fig:3-by-3-variables}
\end{figure}



\begin{defn}
\label{defn:coherence-equation}
We call a \CM frieze $\mathbf{z}$ \emph{coherent} if, for all $i,j \in \ZZ^2$ with~$2 \leq j-i \leq n-2$, we have
\begin{equation}
\label{eq:coherence}
\partial_\leftarrow M(\xdiamond_{i-1,j-1}(\mathbf{z}))\,
\partial_\rightarrow M(\xdiamond_{i,j}(\mathbf{z}))
=
\partial_\uparrow M(\xdiamond_{i-1,j}(\mathbf{z}))\,
\partial_\downarrow M(\xdiamond_{i,j-1}(\mathbf{z})). 
\end{equation}
Accordingly, we call \eqref{eq:coherence} (or \eqref{eq:hexagon-coherence}) the \emph{coherence condition}. 
\end{defn}

\begin{thm}
\label{thm:geometric-frieze-is-coherent}
For any polygon $P$ in the plane~$\AAA$, the corresponding \CM frieze $\mathbf{z}_{\textup{CM}}(P)$ (see Definition~\ref{defn:CM-Frieze-from-P}) is coherent. 
\end{thm}

\begin{proof}
The coherence condition~\eqref{eq:coherence} for $(i,j) \in \ZZ^2$ is precisely equation~\eqref{eq:hexagon-coherence} for the (possibly degenerate) sub-hexagon $(A_{i-1},A_i,A_{i+1},A_{j-1},A_j,A_{j+1})$ of~$P$. 
This equation holds by virtue of Proposition~\ref{prop:(non)coh-equations}.
%
%
%
\end{proof}

\begin{rem}
\label{rem:coherence-13}
The coherence condition~\eqref{eq:coherence} involves 13~entries of the frieze whose indices are shown in Figure~\ref{fig:3-by-3-variables}. 
The indexing set includes 9~integer points forming the $3\times 3$ grid 
$\{i-1,i,i+1\}\times\{j-1,j,j+1\}$ together with 4~indices 
$\{(i\pm\frac12,\neline), (\seline,j\pm\frac12)\}$ corresponding to slanted dashed lines. To write the coherence condition in a more explicit (but not too bulky) form, we introduce the temporary notation 
\begin{align*}
\left[{\ }
\begin{matrix}
       &  d_2  &  t_{13} \\[8pt]
d_1 &  t_{12} &         & t_{23} \\[8pt]
t_{11} &        & t_{22} &        & t_{33} \\[8pt]
b_1  & t_{21} &    & t_{32} \\[8pt]
      & b_2  &   t_{31}
\end{matrix}
{\,}\right]
\stackrel{\rm def}{=}
\left[{\ }
\begin{matrix}
       &  z_{(\smallseline,j+\frac12)}  &  z_{(i-1,j+1)} \\[8pt]
z_{(\smallseline,j-\frac12)} &  z_{(i-1,j)} &         & z_{(i,j+1)} \\[8pt]
z_{(i-1,j-1)} &        & z_{(i,j)} &        & z_{(i+1,j+1)} \\[8pt]
z_{(i-\frac12,\smallneline)}  & z_{(i,j-1)} &    & z_{(i+1,j)} \\[8pt]
      &  z_{(i+\frac12,\smallneline)} &   z_{(i+1,j-1)}
\end{matrix}
{\,}\right].
\end{align*}
Using this notation along with~\eqref{eq:Q(abcdef)}, the coherence condition~\eqref{eq:coherence} becomes
\begin{align}
\label{eq:coherence-Q}
&Q(t_{23},b_2,t_{32},d_2,t_{22},t_{33}) \,Q(t_{21}, b_1, t_{12}, d_1, t_{22}, t_{11}) \\
\notag
=\,&Q(t_{12},d_2,t_{23},b_1,t_{22},t_{13})\,Q(t_{32},d_1,t_{21}, b_2,t_{22},t_{31}) 
\end{align}
(cf.\ \eqref{eq:coherence-Q-123456}). 
\end{rem}

\begin{example}
\label{example:noncoherent}
The \CM frieze $\mathbf{z}_{\textup{CM}}(P)$ shown in Figure~\ref{fig:coherent-frieze-example} is obtained from the polygon~$P$ in Figure~\ref{fig:nonconvex-hexagon}, and is therefore coherent by Theorem~\ref{thm:geometric-frieze-is-coherent}. 
 
Figure~\ref{fig:noncoh-frieze} shows a non-coherent \CM frieze $\mathbf{z}$ of order~6. 
\end{example}

\begin{figure}[ht]
\begin{center}
\vspace{-.1in}
\includegraphics[scale=0.65]{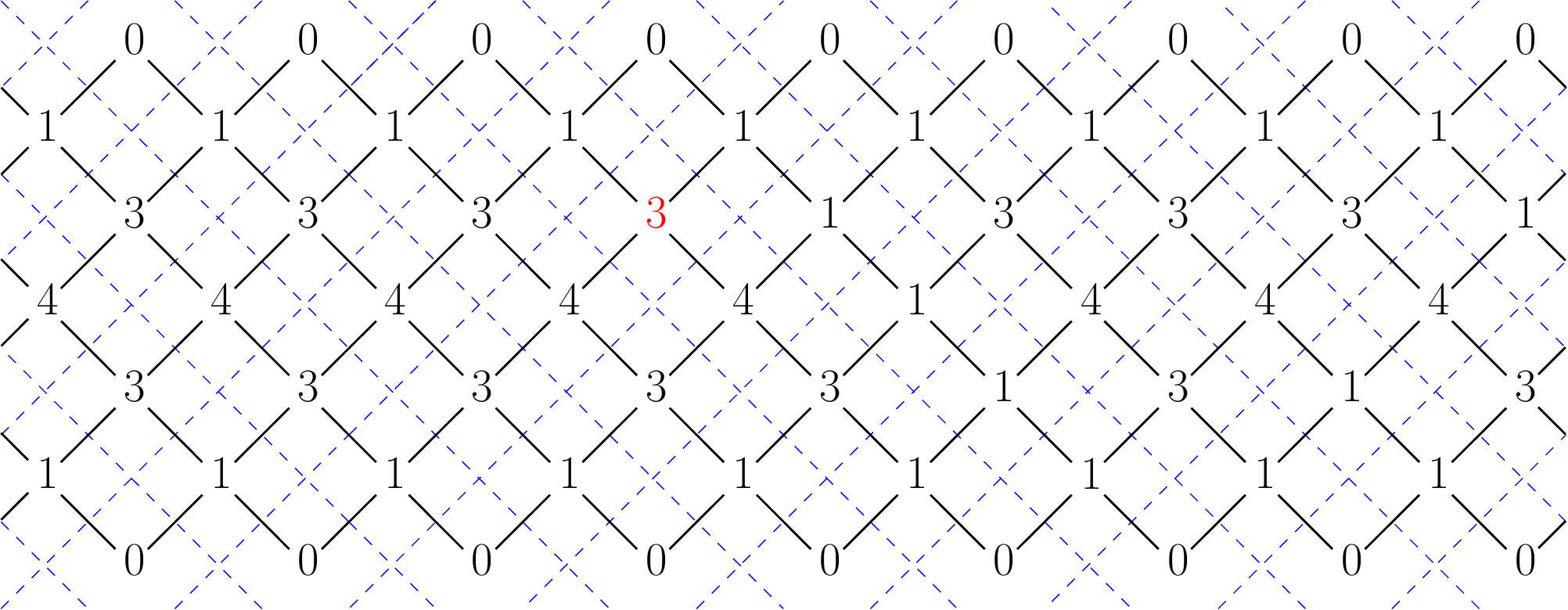}
\end{center}
\caption{A non-coherent \CM frieze $\mathbf{z}$ of order~6. The entries asso\-ciated with the dashed lines (which match the entries in the top and bottom nonzero rows) 
are all equal to~1. 
The coherence condition is violated at one place only, namely for the four diamonds surrounding the red~entry. 
%
}
\vspace{-.2in}
\label{fig:noncoh-frieze}
\end{figure}

%


The following result shows that the coherence condition can be used as a basis for a rational recurrence.

\begin{prop}
\label{prop:(non)coh-equations}
Let 
\begin{align}
\label{eq:x1-hexagon-again}
\mathbf x_1 &= (x_{15},x_{12},x_{24},x_{45},x_{14},x_{25}),\\
\label{eq:x2-hexagon-again}
\mathbf x_2 &= (x_{16},x_{12},x_{25},x_{56},x_{15},x_{26}),\\
\label{eq:x3-hexagon-again}
\mathbf x_3 &= (x_{25},x_{23},x_{34},x_{45},x_{24},x_{35})
\end{align}
be \CM diamonds, cf.\ \eqref{eq:x1-hexagon}--\eqref{eq:x3-hexagon} and Figure~\ref{fig:3-by-3-variables-hexagon}. 
If $x_{25}H_{245}H_{125} \neq 0$, then there is a unique number $x_{36} \in \CC$ such that equation~\eqref{eq:hexagon-coherence} is satisfied, where
\begin{equation}
\label{eq:x4-hexagon-again}
\mathbf x_4 = (x_{26},x_{23},x_{35},x_{56},x_{25},x_{36}), 
\end{equation}
as in~\eqref{eq:x4-hexagon}. Moreover, $\mathbf x_4$ is a \CM diamond.

Similarly, let $\mathbf{x}_2$, $\mathbf{x}_3$, $\mathbf{x}_4$ be three \CM diamonds as in \eqref{eq:x2-hexagon-again}--\eqref{eq:x4-hexagon-again}. If~$x_{25}H_{256}H_{235} \neq 0$, then there is a unique number $x_{14}\in\CC$ such that equation~\eqref{eq:hexagon-coherence} is satisfied, with $\mathbf{x}_1$ given by \eqref{eq:x1-hexagon-again}. Moreover, $\mathbf x_1$ is a \CM diamond.
\end{prop}

\begin{proof}
Direct inspection shows that equation \eqref{eq:hexagon-coherence} is of degree~$1$ in the variable~$x_{14}$, with the coefficient of~$x_{14}$ being $x_{25}\partial_\rightarrow M(\mathbf x_4)$. This coefficient is nonzero because
\[
x_{25}^2(\partial_\rightarrow M(\mathbf x_4))^2=x_{25}^2\cdot 4H_{256}H_{235} \neq0. 
\]
Hence the solution exists and is unique. The case of $x_{36}$ is treated analogously. 

Let us now prove that $\mathbf{x}_1$ is a \CM diamond, given that the same is true about $\mathbf{x}_2$, $\mathbf{x}_3$, and $\mathbf{x}_4$. Applying Lemma~\ref{lem:noncoh-discriminant} to the \CM diamonds $\mathbf{x}_2$, $\mathbf{x}_3$, and $\mathbf{x}_4$, and the identity~\eqref{eq:M-and-partial-M} to~$\mathbf{x}_1$, we obtain:
\begin{align}
\label{eq:two-partials-1}
\bigl(\partial_\uparrow M(\mathbf x_2)\partial_\downarrow M(\mathbf x_3)\bigr)^2
&=16H_{245}H_{125}H_{256}H_{235}, \\
\label{eq:two-partials-2}
\bigl(\partial_\leftarrow M(\mathbf x_1)\partial_\rightarrow M(\mathbf x_4)\bigr)^2 
&= (4H_{245}H_{125} - 8 x_{25} M(\mathbf{x}_1))\cdot 4H_{256}H_{235}.
\end{align}
Since \eqref{eq:hexagon-coherence} holds, the expressions on the left-hand-sides of \eqref{eq:two-partials-1} and \eqref{eq:two-partials-2} are equal to each other, and we conclude that $x_{25} H_{256}H_{235} M(\mathbf{x}_1)=0$. Given that $x_{25} H_{256}H_{235}\neq 0$, we get $M(\mathbf{x}_1)=0$, as desired. The case of~$\mathbf{x}_4$ is similar. 
\end{proof}

\begin{defn}[cf.\ Definition~\ref{defn:traversing-Heronian}]
\label{defn:traversing-CM}
A \emph{traversing path} $\pi$ for an order~$n$ \CM frieze is an ordered collection 
\[
\pi=((i_1, j_1), \dots, (i_{n-1},j_{n-1}),\ell_1, \dots, \ell_{n-2})
\]
of $2n-3$ indices in $\InCM$ such that
\begin{itemize}
\item
$(i_1,j_1), \dots, (i_{n-1},j_{n-1})$ are integer points in $\{(i,j)\in\ZZ^2 : 1\le j-i\le n-1\}$; 
\item
$\ell_1, \dots, \ell_{n-2}$ are lines in $L_n$; 
\item
$j_k - i_k = k$ for $k=1,\dots,n-1$;
\item
$|i_{k+1} - i_k| + |j_{k+1} - j_k| = 1$, for $k=1,\dots,n-2$;
\item
if $j_{k+1}=j_k$, then $\ell_k = (i_k-\frac12,\neline)\in L_n$; 
\item
if $i_{k+1}=i_k$, then $\ell_k = (\seline, j_k+\frac12)\in L_n$.
\end{itemize}
The following less formal description is perhaps more illuminating. Let us view the set $\{(i,j)\in\ZZ^2 : 1\le j-i\le n-1\}$ as the vertex set of a graph (a two-dimensional integer lattice). Then: 
\begin{itemize}
\item[(a)]
$(i_1,j_1), \dots, (i_{n-1},j_{n-1})$ are the nodes lying on a shortest path connecting the lower and upper boundaries of the strip of interior nodes; 
\item[(b)]
$\ell_1, \dots, \ell_{n-2}$ are the dashed lines intersecting this shortest path.  \end{itemize}
For a traversing path as above, we call the collection of $3n-4$ indices 
\[
\overline\pi=((i_1, j_1), \dots, (i_{n-1},j_{n-1}),(i_1+1, j_1+1), \dots, (i_{n-1}+1,j_{n-1}+1),\ell_1, \dots, \ell_{n-2})
\]
the \emph{thickening} of~$\pi$. 
Thus the \emph{thickened path} $\overline\pi$ consists of the subsets (a) and (b) described above together with
\begin{itemize}
\item[(c)]
the nodes on the path (a) shifted by $(1,1)$ to the right. 
\end{itemize}
\end{defn}

\begin{example}
\label{example:CM-traversing-path}
Figure~\ref{fig:traversing-path-example} shows the traversing path
\[
\pi=\bigl((0,1),(0,2),(-1,2),(-1,3),(\seline,\tfrac32), (-\tfrac12,\neline), (\seline,\tfrac52)\bigr).
\]
for an order 5 \CM frieze, cf.\ Example~\ref{example:heronian-traversing-path}. 
Its thickening $\overline\pi$ is given by  
\[
\overline\pi=\bigl((0,1),(0,2),(-1,2),(-1,3),(1,2),(1,3),(0,3),(2,3),(\seline,\tfrac32), (-\tfrac12,\neline), (\seline,\tfrac52)\bigr).
\]
The paths $\pi$ and $\overline\pi$ include $2n - 3 = 7$ and $3n - 4 = 11$ indices, respectively. 
\end{example}

\begin{figure}[ht]
\begin{center}
\includegraphics[scale=0.8]{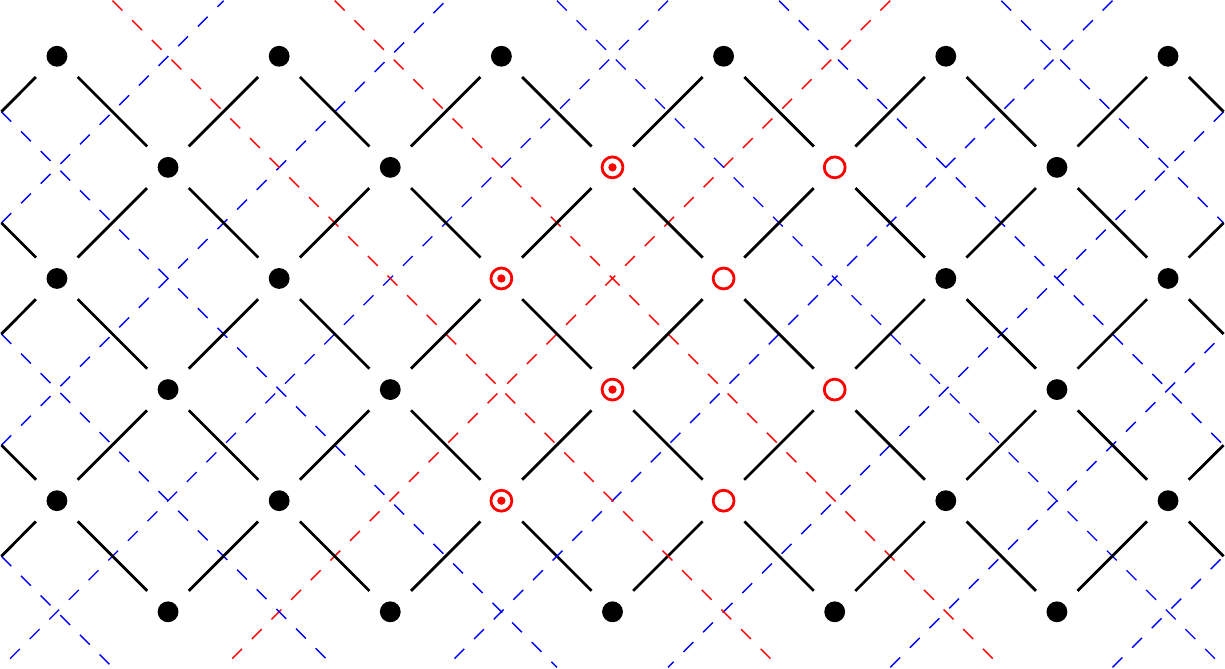}
\end{center}
\caption{A traversing path~$\pi$ for a \CM frieze of order $n=5$, and its thickening~$\overline\pi$, see   Example~\ref{example:CM-traversing-path}. The dashed lines in $\pi$ are colored red. The nodes in $\pi$ are the circled red nodes; the ones in $\overline\pi \setminus \pi$ are hollow red.}
\vspace{-.2in}
\label{fig:traversing-path-example}
\end{figure}

\begin{thm}[{\rm cf.\ Corollary~\ref{cor:frieze-is-recurrence}}]
\label{thm:CM-frieze-is-recurrence}
Let $\mathbf{z}\!=\! (z_\alpha)_{\alpha \in \InCM}$ be a coherent \CM frieze of order~$n$ such that 
\begin{align}
\label{eq:CM-nonzero-interior-z}
&\quad z_{(i,j)}\neq 0 \ \  \text{for $(i,j)\in \ZZ^2, \ 2\le j-i\le n-2$,}\\
\label{eq:CM-nonzero-interior-H}
&\begin{cases}H(z_{(i,j)},z_{(i+1,j)},z_{(i+\frac12,\smallneline)}) \neq 0, \\
H(z_{(i,j-1)},z_{(i,j)},z_{(\smallseline,j-\frac12)}) \neq 0
\end{cases}
\text{for $(i,j)\in \ZZ^2,\  2\le j-i\le n-1$}. 
\end{align}
Then $\mathbf{z}$ is uniquely determined by its entries belonging to the thickening $\overline \pi$ of an arbitrary traversing path~$\pi$. 
\end{thm}

\begin{proof}
Given the entries indexed by the elements of~$\overline\pi$, 
the boundary conditions \eqref{eq:cm-boundary-1}--\eqref{eq:cm-boundary-2} allow us to determine the entries of $\mathbf z$ indexed by the lines~$L_n$ as well as those indexed by the four rows 
$\{(i,j) \!\in\! \ZZ^2\colon j-i \!\in\! \{0,1,n\!-\!1,n\}\}$. 
To reconstruct~ the remaining entries, indexed by \hbox{$\{(i,j)\! \in\! \ZZ^2\colon 2 \!\leq\! j\!-\!i \!\leq\! n\!-\!2\}$}, we repeatedly use the co\-herence equation~\eqref{eq:coherence} to propagate away from~$\overline\pi$. 
Proposition~\ref{prop:(non)coh-equations} ensures both the existence and the uniqueness of propagation, so the resulting frieze agrees~with~$\mathbf z$.
\end{proof}

\begin{rem}
It is instructive to make a comparison between the assumptions under\-lying Corollary~\ref{cor:frieze-is-recurrence} and Theorem~\ref{thm:CM-frieze-is-recurrence}, or equivalently the corresponding recursive algorithms for constructing Heronian and \CM friezes. Corollary~\ref{cor:frieze-is-recurrence} relies on nonvanishing at the interior integer points, see \eqref{eq:nonzero-interior}/\eqref{eq:CM-nonzero-interior-z}. (In a geometric setting, the squared lengths of diagonals must be nonzero.) Theorem~\ref{thm:CM-frieze-is-recurrence} needs the additional requirement \eqref{eq:CM-nonzero-interior-H}: the nonvanishing of the Heron expressions. (In~a geometric setting, this means that certain triangles must have nonzero areas.) 
In other words, (re)constructing a Heronian frieze is computationally more feasible  than the similar task for a \CM frieze. 
\end{rem}

The following result can be viewed as a partial converse to Theorem~\ref{thm:geometric-frieze-is-coherent}. 

\begin{thm}
\label{thm:CM-frieze-glide-symmetry}
Let $\mathbf{z}_{\textup{CM}}\!=\! (z_\alpha)_{\alpha \in \InCM}$ be a coherent \CM frieze of order~$n$ satisfying the conditions in Theorem~\ref{thm:CM-frieze-is-recurrence}. 
Then there exists a plane $n$-gon~$P$ such that $\mathbf{z}_{\textup{CM}}\!=\!\mathbf{z}_{\textup{CM}}(P)$, cf.\ Definition~\ref{defn:CM-Frieze-from-P}. 
Consequently $\mathbf{z}_{\textup{CM}}$ has the glide symmetry: 
\[
z_{(i,j)} = z_{(j,i+n)} \quad (1\le j-i\le n-1). 
\]
\end{thm}

Theorem~\ref{thm:CM-frieze-glide-symmetry} will be proved at the end of Section~\ref{sec:CM-vs-heronian}. 

\begin{rem}
The nonvanishing conditions \eqref{eq:CM-nonzero-interior-z}--\eqref{eq:CM-nonzero-interior-H} appearing in Theorems \ref{thm:CM-frieze-is-recurrence} and~\ref{thm:CM-frieze-glide-symmetry}
are satisfied by any \CM frieze $\mathbf{z}_{\textup{CM}}(P)$ associated with a polygon~$P$ in the \emph{real} plane~$\RR^2$ such that any line extending a side of~$P$ does not pass through a third vertex. 
This condition is violated for the polygon shown in Figure~\ref{fig:nonconvex-hexagon}, so condition~\eqref{eq:CM-nonzero-interior-H} fails for the coherent frieze $\mathbf{z}_{\textup{CM}}(P)$ shown in Figure~\ref{fig:coherent-frieze-example}. 
\end{rem}

\pagebreak[3]

\section{\CM friezes vs.\ Heronian friezes}
\label{sec:CM-vs-heronian}

Our first goal is to show that, under mild genericity conditions, a Heronian diamond restricts to a \CM diamond. 





\begin{lem}
\label{lem:heronian-partial-menger}
Let $\xS = (a,b,c,d,e,f,p,q,r,s) \in \CC^{10}$ be a Heronian diamond satisfying the following condition: 
\begin{equation}
\label{eq:partial-menger-conditions}
(e,f) \neq (0,0) \quad {\underbar{\rm or}} \quad a = q = r = 0  \quad {\underbar{\rm or}} \quad c = p = s = 0.
\end{equation}
Then 
\begin{alignat}{3}
\label{eq:partial-xe,xf}
-2rs &= \partial_\leftarrow M(a,b,c,d,e,f), \qquad & -2pq &= \partial_\rightarrow M(a,b,c,d,e,f),\\
\label{eq:partial-xb,xd}
2qs &= \partial_\smallneline M(a,b,c,d,e,f), \qquad &2pr &= \partial_\smallseline M(a,b,c,d,e,f),\\
\label{eq:partial-xa,xc}
2ps &= \partial_\uparrow M(a,b,c,d,e,f), \qquad &2rq &= \partial_\downarrow M(a,b,c,d,e,f).
\end{alignat}
\end{lem}

\begin{proof}
If $a \!=\! q\! =\! r \!=\! 0$ or $c \!= \!p \!= \!s\! = \!0$, then formulas~\eqref{eq:partial-xe,xf}--\eqref{eq:partial-xa,xc} can be checked one by one, taking care to apply Lemmas~\ref{lem:diamond-0-top} or~\ref{lem:diamond-0-bottom}, respectively. For example, $a \!=\! q \!=\! r\! =\! 0$ implies $d=e$, $f=b$, $ps=p^2=H(b,c,e)$, and consequently $\partial_\uparrow M(a,b,c,d,e,f)=2(e-b)(f-d)+2c(e+d+f+b-c-2a) 
=-2(e-b)^2+2c(2e+2b-c)=2H(b,c,e)=2ps$. 

If $(e,f) \neq (0,0)$, then we can assume that $e \neq 0$, since the case $f\neq 0$ can~be treated in the same way. Now Corollary~\ref{cor:triangulation-recovery} (applied to the triangulated 4-cycle~$G$ with diagonal~$\{1,3\}$) and Proposition~\ref{prop:heron-propagates} imply that $\xS = \xS(P)$ for some plane quadrilateral~$P$. 
It remains to apply Proposition~\ref{prop:partial-menger-formulas} and observe that 
equations~\eqref{eq:partial-xe,xf}--\eqref{eq:partial-xa,xc} are a restatement of \eqref{eq:partial-13,24}--\eqref{eq:partial-14,23} via the notational conventions \eqref{eq:abcd-ABCD}--\eqref{eq:pqrs-ABCD}.
\end{proof}

\begin{prop}
\label{prop:heronian-diamond-to-CM}
Let $(a,b,c,d,e,f,p,q,r,s) \in \CC^{10}$ be a Heronian diamond satisfying condition~\eqref{eq:partial-menger-conditions}. Then $(a,b,c,d,e,f)$ is a \CM diamond.
\end{prop}

\begin{proof}
First suppose $a = q = r = 0$ or $c = p = s = 0$. By Lemmas~\ref{lem:diamond-0-top} and~\ref{lem:diamond-0-bottom}, it remains to check that $(0,b,c,e,e,b)$ and $(a,b,0,d,b,d)$ are \CM diamonds. This can be verified by direct computation.

Now suppose $(e,f) \neq 0$. Without loss of generality, we can assume that $e\neq0$ (cf.\ Proposition~\ref{prop:diamond-reflect} and Remark~\ref{rem:CM-dihedral}). 
Equations~\eqref{eq:r=}--\eqref{eq:s=} imply that 
\[
2ers = ps(e+a-d)+qs(e-c+b) = pr(e-a+d)+qr(e+c-b),
\]
or equivalently
\begin{equation}
\label{eq:heron-to-CM-start}
2ers - \frac12\B((a-d)(ps-pr)+(b-c)(qs-qr)+e(ps+qs+pr+qr)\B) = 0.
\end{equation}
Let~$\mathbf x = (a,b,c,d,e,f)$. Substituting the expressions for $pq$, $rs$, $pr$, $qs$, $rq$, $ps$ given in Lemma~\ref{lem:heronian-partial-menger} into~\eqref{eq:heron-to-CM-start} and negating both sides results in
\begin{align*}
\frac e4\B(4\partial_\rightarrow M(\mathbf x)&+\partial_\uparrow M(\mathbf x)+\partial_\smallneline M(\mathbf x)+\partial_\downarrow M(\mathbf x)+\partial_\smallseline M(\mathbf x)\B) \\&+ \frac{a-d}4\B(\partial_\uparrow M(\mathbf x) - \partial_\smallseline M(\mathbf x)\B) + \frac{b-c}4\B(\partial_\smallneline M(\mathbf x) - \partial_\downarrow M(\mathbf x)\B) = 0. 
\end{align*}
The left-hand side of the last equation is nothing but~$M(\mathbf x)$.
\end{proof}

It turns out that when we restrict from the Heronian setup to the Cayley-Menger one, the coherence condition~\eqref{eq:coherence} is automatically satisfied:  


\begin{prop}
\label{prop:heronian-coherent-agree}
Let~$\mathbf x_1$,~$\mathbf x_2$,~$\mathbf x_3$,~$\mathbf x_4$ be four  Heronian diamonds arranged in a grid, as shown in Figure~\ref{fig:3-by-3-heronian-variables}. Suppose that each $\mathbf x_i$ satisfies condition~\eqref{eq:partial-menger-conditions}. Then the corresponding \CM diamonds (cf.\ Proposition~\ref{prop:heronian-diamond-to-CM}) satisfy the coherence equation~\eqref{eq:coherence}.
\end{prop}

\begin{figure}[ht]
\begin{center}
\vspace{-.1in}
\includegraphics[scale=1.2]{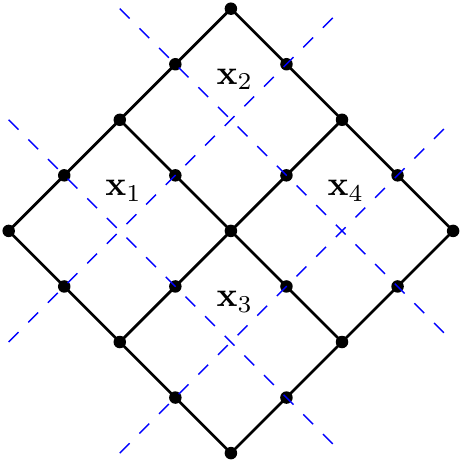}
\vspace{-.1in}
\end{center}
\caption{The arrangement of four Heronian diamonds
in Proposition~\ref{prop:heronian-coherent-agree}.}
\vspace{-.2in}
\label{fig:3-by-3-heronian-variables}
\end{figure}

\begin{proof}
We may apply Lemma~\ref{lem:heronian-partial-menger} to each diamond. Formulas~\eqref{eq:partial-xe,xf}--\eqref{eq:partial-xa,xc} imply that both sides of the coherence equation~\eqref{eq:coherence} will be equal to the product of the four entries adjacent to the central node in Figure~\ref{fig:3-by-3-heronian-variables}.
\end{proof}

\begin{prop}
\label{prop:recover-S}
Let $\mathbf x = (a,b,c,d,e,f)$ be a \CM diamond such that $(e,f) \neq (0,0)$ and
\begin{eqnarray}
\label{eq:HHHHnot0}
&H(b,c,e)H(a,d,e)H(a,f,b)H(c,f,d) \neq 0. 
\end{eqnarray}
Then there exist exactly two Heronian diamonds which restrict to $\mathbf x$, differing from each other by a simultaneous sign change of~$\{p,q,r,s\}$.
\end{prop}

\begin{proof}
In view of Lemma~\ref{lem:noncoh-discriminant} and condition~\eqref{eq:HHHHnot0}, the partial derivatives appearing in \eqref{eq:partial-xe,xf} 
do not vanish. Hence for any Heronian diamond $(a,b,c,d,e,f,p,q,r,s)$ restricting to~$\mathbf x$, we have $pqrs\neq 0$. 
Moreover the pairwise products of the nonzero numbers $p,q,r,s$ must be given by~\eqref{eq:partial-xe,xf}--\eqref{eq:partial-xa,xc}. It follows that these numbers are  uniquely determined by~$\mathbf{x}$, up to a simultaneous sign change. 

It remains to show existence. 
We need to check that conditions 
\[
M(\mathbf x) = 0, \quad 
p^2 = H(b,c,e), {\ \ }
q = -\tfrac1{2p}\partial_\rightarrow M(\mathbf x), {\ \ }
r = \tfrac1{2p}\partial_\smallseline M(\mathbf x), {\ \ }
s = \tfrac1{2p}\partial_\uparrow M(\mathbf x)
\]
imply the requirements \eqref{eq:heron-p}--\eqref{eq:bilinear} from the definition of a Heronian diamond. 
The Heron relations~\eqref{eq:heron-q}--\eqref{eq:heron-s} follow using 
\eqref{eq:noncoh-f}--\eqref{eq:noncoh-e}. Since~$p \neq 0$, equation~\eqref{eq:additivity} is equivalent to
$2H(b,c,e) - \partial_\rightarrow M(\mathbf x) = \partial_\smallseline M(\mathbf x) + \partial_\uparrow M(\mathbf x)$,
which can be verified by a direct calculation. Finally, equations~\eqref{eq:bretschneider} and~\eqref{eq:bilinear} reduce to $M(\mathbf x) \!=\! 0$ after eliminating $p,q,r,s$.
\end{proof}

We are now ready to describe the relationship between Heronian and \CM friezes. 

\begin{thm}
\label{thm:heronian-iff-coherent} 
Let $\mathbf{z}= (z_\alpha)_{\alpha \in I_n}$ be a Heronian frieze such that  
\begin{equation}
\label{eq:nonzero-interior-again}
\text{$z_{(i,j)}\neq 0$ for any $(i,j)\in\ZZ^2$  such that  $2\le j-i\le n-2$.}
\end{equation}
Then the restriction 
of $\mathbf{z}$ to $\InCM$ is a coherent \CM~frieze. 

$\!$Conversely, let $\mathbf{z}_{\textup{CM}} \!=\! (z_\alpha)_{\alpha \in \InCM}$ be a coherent \CM frieze satisfying~\eqref{eq:nonzero-interior-again}. In addition, suppose that for every $(i,j) \in \ZZ^2$ with~$2 \leq j-i \leq n-1$, we have
\begin{equation}
\label{eq:nonzero-heron-quantities}
H(z_{(i,j)},z_{(i+1,j)},z_{(i+\frac12,\smallneline)}) \neq 0, {\ \ } H(z_{(i,j-1)},z_{(i,j)},z_{(\smallseline,j-\frac12)}) \neq 0. 
\end{equation}
Then there exists a Heronian frieze $\mathbf{z} = (z_\alpha)_{\alpha \in I_n}$ which extends~$\mathbf{z}_{\textup{CM}}$. This extension is unique up to a global change of sign of the entries indexed by~$I_n \setminus \InCM$. 
\end{thm}

We note that condition~\eqref{eq:nonzero-interior-again} is the same as \eqref{eq:nonzero-interior} or \eqref{eq:CM-nonzero-interior-z}, whereas condition \eqref{eq:nonzero-heron-quantities} is the same as~\eqref{eq:CM-nonzero-interior-H}. 

\begin{proof}
The restriction of $\mathbf{z}$ to $\InCM$ is a \CM frieze by Proposition~\ref{prop:heronian-diamond-to-CM}. By Corollary~\ref{cor:frieze-to-polygon}, $\mathbf{z} = \mathbf{z}(P)$ is the frieze obtained from some $n$-gon, so Theorem~\ref{thm:geometric-frieze-is-coherent} implies that $\mathbf{z}_{\textup{CM}} = \mathbf{z}_{\textup{CM}}(P)$ is coherent. This proves the first part of the theorem.

Now let $\mathbf{z}_{\textup{CM}}$ be a coherent \CM frieze satisfying~\eqref{eq:nonzero-interior-again}--\eqref{eq:nonzero-heron-quantities}. Let $\pi$ be a traversing path in $\InCM$, and let $\overline\pi$ be its thickening, see Definition~\ref{defn:traversing-CM}. 
Let $\widetilde{\pi}$ be the ``lift'' of~$\overline\pi$ to~$I_n$, constructed from two adjacent traversing paths in $I_n$ whose restriction onto $\InCM$ agrees with~$\overline\pi$; in addition, $\widetilde{\pi}$~contains the midpoints of lattice segments connecting neighboring points lying on these two paths to each other. Successively applying Proposition~\ref{prop:recover-S} to the string of \CM diamonds whose indexing sets are contained in~$\overline\pi$, we conclude that there exist exactly two arrays $\widetilde{\mathbf{z}}_1$~and $\widetilde{\mathbf{z}}_2$ on $\widetilde{\pi}$ which agree with 
$\mathbf{z}_{\textup{CM}}$ along $\overline\pi$, and which satisfy the Heronian diamond equations for each diamond in~$\widetilde{\pi}$. Moreover, these arrays differ by a simultaneous sign change of the entries indexed by~$\widetilde{\pi}\setminus\overline\pi$. To complete the proof of the theorem, it remains to establish the following claims:
\begin{itemize}[leftmargin=.3in]
\item[{(i)}]
there exist unique Heronian friezes $\mathbf z_1$ and $\mathbf z_2$ which extend $\widetilde{\mathbf z}_1$ and $\widetilde{\mathbf z}_2$, respectively, and restrict to~$\mathbf{z}_{\textup{CM}}$;
\item[{(ii)}]
${\mathbf z}_1$ and ${\mathbf z}_2$ differ by a global sign change of the entries indexed by~$I_n\setminus \InCM$.
\end{itemize}

Let~$\mathbf x_1$,~$\mathbf x_2$,~$\mathbf x_3$ be three Heronian diamonds located along~$\overline\pi$, all sharing a node~$z_{(i,j)}$. Apply Corollary~\ref{cor:heronian-propagate} to construct the fourth Heronian diamond $\mathbf x_4$ containing~$z_{(i,j)}$. The boundary conditions \eqref{eq:cm-boundary-1}--\eqref{eq:cm-boundary-2} required of a \CM frieze, together with~\eqref{eq:nonzero-interior-again}--\eqref{eq:nonzero-heron-quantities} along $\overline\pi$, ensure that each of the four Heronian diamonds $\mathbf{x}_1, \mathbf{x}_2, \mathbf{x}_3, \mathbf{x}_4$ satisfies condition~\eqref{eq:partial-menger-conditions}. By Proposition~\ref{prop:heronian-coherent-agree}, this establishes the coherence condition~\eqref{eq:coherence} for the corresponding four \CM diamonds. 
Furthermore, Proposition~\ref{prop:(non)coh-equations} applies (thanks to~\eqref{eq:nonzero-interior-again}--\eqref{eq:nonzero-heron-quantities}), implying that the newly constructed entry of $\mathbf x_4$ agrees with the corresponding entry of~$\mathbf z_{\textup{CM}}$. We now repeat the above process over and over, propagating away from~$\overline\pi$, to construct the (unique) Heronian friezes ${\mathbf z}_1$ and ${\mathbf z}_2$ satisfying the specifications in claim~(i). Near the boundary, we use propagation rules for Heronian friezes (see \eqref{eq:d=e}--\eqref{eq:s=p} and \eqref{eq:b=e}--\eqref{eq:r=q}), which agree with their counterparts for \CM friezes, see~\eqref{eq:cm-boundary-1}--\eqref{eq:cm-boundary-2}. 

Finally, a repeated application of Proposition~\ref{prop:recover-S} establishes claim~(ii). 
\end{proof}


\begin{proof}[Proof of Theorem~\ref{thm:CM-frieze-glide-symmetry}]
By Theorem~\ref{thm:heronian-iff-coherent}, the \CM frieze $\mathbf{z}_{\textup{CM}}$ can be extended to a Heronian frieze. The latter comes from a polygon by Corollary~\ref{cor:frieze-to-polygon}. Hence so does the former.
\end{proof}


\section*{Acknowledgments}

This work was conducted in Summer 2019 within the framework of the REU program at the University of Michigan. 
We thank Igor Dolgachev, Steven Gortler,~and Dylan Thurston for answering some questions that arose in our work. 
We are grateful to the referees for a number of suggestions which improved the quality of exposition. 

\bibliographystyle{amsplain}

\end{document}